\documentclass[10pt]{article}

\usepackage{amsmath,amsfonts,amssymb,amsthm,graphics,amscd}
\usepackage{latexsym,mathrsfs}

\usepackage{color}

\newtheorem{theorem}{Theorem}[section]
\newtheorem{lemma}[theorem]{Lemma}
\newtheorem{corollary}[theorem]{Corollary}

\newtheorem{example}[theorem]{Example}
\newtheorem{definition}[theorem]{Definition}
\newtheorem{proposition}[theorem]{Proposition}
\newtheorem{remark}[theorem]{Remark}

\def\ZZ{\mathbb{Z}}

\def\NN{\mathbb{N}}
\def\CC{\mathbb{C}}

\def\DD{\mathbb{D}}

\numberwithin{equation}{section}
\newcommand{\bex}{\begin{example}\rm}
\newcommand{\eex}{\end{example}}

\def\inter{{\rm int}}

\def\acc{{\rm acc}}
\def\iso{{\rm iso}}
\def\dim{{\rm dim}}

\def\ind{{\rm ind}}

\def\codim{{\rm codim}\, }
\def\NN{{\mathbb N}}
\def\ZZ{{\mathbb Z}}
\def\CC{{\mathbb C}}

\def\R{{\bf R}}
\def\DD{\mathbb{D}}
\def\W{\mathcal{W}}
\def\B{\mathcal{B}}

\def\S{\mathcal{S}}
\def\I{\mathcal{I}}
\def\DDD{\bf{D}}
\def\BB{\bf B}
\def\ds{\displaystyle}
\def\p{\partial}

\begin{document}

\title{
Topological uniform descent, quasi-Fredholmness  and operators  originated from semi-B-Fredholm theory}

\author {Sne\v zana  \v{C}. \v{Z}ivkovi\'{c}-Zlatanovi\'{c}\footnote{The author is
supported by the Ministry of Education, Science and Technological
Development, Republic of Serbia, grant no. 174007.}, \ Mohammed Berkani}

\date{}

\maketitle
\setcounter{page}{1}

\begin{abstract}

 In this paper we study  operators originated from semi-B-Fredholm theory and as a consequence  we get some results regarding boundaries and connected hulls of the corresponding spectra.  In particular, we prove that a bounded linear operator $T$ acting on a Banach space, having  topological uniform descent,
 is a {\bf BR} operator if and only if  $0$ is not an accumulation point of the associated spectrum   $\sigma_{\bf R}(T)=\{\lambda\in\CC:T-\lambda I\notin {\bf R}\}$,  where
 ${\bf R}$ denote any of the following classes: upper
 semi-Weyl operators, Weyl operators, upper
 semi-Fredholm operators, Fredholm operators, 
 operators with finite (essential) descent and  ${\bf BR}$ the B-regularity associated to
 ${\bf R}$ as in \cite{P8}. Under the stronger  hypothesis of quasi-Fredholmness of $T,$   we obtain a similar characterisation for $T$ being
  a {$\bf BR$}  operator for much larger  families of sets ${\bf R}.$

 \end{abstract}

2010 {\it Mathematics subject classification\/}:  47A53, 47A10.

{\it Key words and phrases\/}: Semi-B-Fredholm operators;  topological uniform descent; quasi-Fredholm operators; boundary; connected hull.

\section{Introduction}

 Let $\mathbb{N} \ (\mathbb{N}_0)$ denote the set of all positive
(non-negative) integers, and let $\mathbb{C}$ denote the set of all
complex numbers.
We use $L(X)$ to denote  the Banach algebra of bounded linear operators acting on  an infinite dimensional complex  Banach space $X$.  The group of all invertible operators is
denoted by $L(X)^{-1}$. Let $\I(X)$ denote the set of all
bounded below operators and let $\S(X)$ denote the set of all
surjective operators. For $T\in L(X)$, denote by  $\sigma(T)$, $\sigma_p(T)$, $\sigma_{ap}(T)$ and $\sigma_{su}(T)$ its  spectrum, point spectrum, approximate point spectrum and surjective spectrum, respectively.  Also, write  $N(T)$ for its null-space, $R(T)$ for its range, $ \alpha (T)$ for its nullity and  $\beta (T)$ for its defect.
{\it The compression spectrum} of $T\in L(X)$, denoted by $\sigma_{cp}(T)$, is the set of all complex $\lambda$ such that $T-\lambda I$ does not have dense range.

An operator $T\in L(X)$ is {\it upper semi-Fredholm} if $ \alpha (T)<\infty\text{
 and
}R(T)\text{ is closed}$,
while $T$ is
{\it lower semi-Fredholm} if $\beta (T)<\infty$. In the sequel   $\Phi_{+}(X)$ (resp.
$\Phi_{-}(X)$)  will denote the set of upper (resp. lower)
semi-Fredholm operators.   If $T$ is upper or lower semi-Fredholm, then $T$ is called {\it semi-Fredholm}. The set of semi-Fredholm operators is denoted by $\Phi_{\pm}(X)$.
              For  semi-Fredholm operators  the index is defined by
           $\ind (T)=\alpha (T)-\beta (T)$. The set of Fredholm operators is defined as
           $
\Phi (X)=  \Phi_+(X) \cap \Phi_-(X).$
The sets  of upper semi-Weyl, lower semi-Weyl and  Weyl operators are
defined as $\W_+(X)=\{ T\in\Phi_+ (X): \ind (T)\le 0\}$, $\W_-(X)=\{ T\in\Phi_- (X): \ind (T)\ge 0\}$  and  $\W(X)=\{ T\in\Phi (X): \ind (T)=0\}$, respectively.

 For $T\in L(X)$, the {\it upper semi-Fredholm spectrum}, the {\it lower semi-Fredholm spectrum}, the {\it  semi-Fredholm spectrum}, the {\it  Fredholm spectrum}, the {\it upper semi-Weyl spectrum}, the {\it lower semi-Weyl spectrum} and the {\it Weyl spectrum} are defined, respectively, by:
 \begin{eqnarray*}
      \sigma_{\Phi_+}(T) &=& \{\lambda\in\CC:T-\lambda I\notin\Phi_+ (X)\}, \\
    \sigma_{\Phi_-}(T) &=& \{\lambda\in\CC:T-\lambda I\notin\Phi_- (X)\},\\
    \sigma_{\Phi_{\pm}}(T) &=& \{\lambda\in\CC:T-\lambda I\notin\Phi _{\pm}(X)\}
    ,\\
    \sigma_{\Phi}(T) &=& \{\lambda\in\CC:T-\lambda I\notin\Phi (X)\},\\
     \sigma_{\W_+}(T) &=& \{\lambda\in\CC:T-\lambda I\notin\W_+ (X)\}, \\
     \sigma_{\W_-}(T) &=& \{\lambda\in\CC:T-\lambda I\notin\W_+ (X)\}, \\
     \sigma_{\W}(T) &=& \{\lambda\in\CC:T-\lambda I\notin\W (X)\}.
 \end{eqnarray*}

For $n\in\NN_0$ we set $c_n(T)=\dim R(T^n)/R(T^{n+1})$ and $c_n^\prime(T)=\dim N(T^{n+1})/N(T^n)$. From \cite[Lemmas 3.1 and 3.2]{Kaashoek} it follows that $c_n(T)=\codim (R(T)+N(T^n))$ and $c_n^\prime(T)=\dim (N(T)\cap R(T^n))$. Obviously, the sequences $(c_n(T))_n$ and $(c_n^\prime(T))_n$ are decreasing. For each $n\in\NN_0$, $T$ induced a linear transformation from the vector space $R(T^n)/R(T^{n+1})$ to the space $R(T^{n+1})/R(T^{n+2})$ and let $k_n(T)$ denote the dimension of the null space of the induced map. From \cite[Lemma 2.3]{Grabiner} it follows that
\[k_n(T)=\dim (R(T^n)\cap N(T))/(R(T^{n+1}) \cap N(T))=\dim (R(T)+N(T^{n+1}))/(R(T)+N(T^n)).\] From this it is easily
seen that
$
k_n(T)=c_n^{\prime}(T)-c_{n+1}^{\prime}(T) $ if $c_{n+1}^{\prime}(T)<\infty$ and $ k_n(T)=c_n(T)-c_{n+1}(T)$ if $c_{n+1}(T)<\infty$.

The  {\it descent} $\delta(T)$ and the {\it ascent} $ a(T) $ of $T$ are defined by
 $ \delta(T)=\inf \{ n\in\NN_0:c_{n}(T)=0 \}=\inf  \{n\in\NN_0: R(T^n) = R(T^{n+1})\}$
and
 $ a(T)=\inf \{ n\in\NN_0:c^\prime_{n}(T)=0 \}= \inf \{n\in\NN_0 : N(T^{n})=N(T^{n+1})\}$. We set formally $\inf\emptyset =\infty$.

The {\it essential  descent} $\delta_e(T)$ and the {\it essential ascent} $ a_e(T) $ of $T$ are defined by
 $ \delta_e(T)=\inf \{ n\in\NN_0:c_{n}(T)<\infty \}$
and
 $ a_e(T)=\inf \{ n\in\NN_0:c^\prime_{n}(T)<\infty \}$.

The sets  of upper semi-Browder, lower semi-Browder and  Browder operators are
defined as $\B_+(X)=\{ T\in\Phi_+ (X): a (T)<\infty\}$, $\B_-(X)=\{ T\in\Phi_- (X): \delta(T)<\infty\}$  and  $\B(X)=\B_+(X)\cap \B_-(X)$, respectively.
 For $T\in L(X)$, the {\it upper semi-Browder spectrum}, the {\it lower semi-Browder spectrum} and the {\it  Browder spectrum} are defined, respectively, by:
 \begin{eqnarray*}
      \sigma_{{\cal B}_+}(T) &=& \{\lambda\in\CC:T-\lambda I\notin\B_+(X)\}, \\
    \sigma_{{\cal B}_-}(T) &=& \{\lambda\in\CC:T-\lambda I\notin\B_-(X)\},\\
    \sigma_{{\cal B}}(T) &=& \{\lambda\in\CC:T-\lambda I\notin\B(X)\}.
\end{eqnarray*}

Sets of {\it left and
right Drazin invertible} operators, respectively, are defined as
 $
LD(X) = \{T\in L(X) : a(T) < \infty {\rm \  and\ } R(T^{a(T)+1})\ {\rm  is\ closed }\}
$
 and
         $
RD(X) = \{T\in L(X) : \delta(T) < \infty {\rm \  and\ } R(T^{\delta(T)})\ {\rm  is\ closed }\}.
$
 If $   a(T) < \infty $ and $\delta(T)<\infty$, then $T$ is called  {\it Drazin invertible} \cite{aienatriolo}, \cite{aienatriolo2}. By $D(X)$ we denote the set of Drazin invertible operators.

An operator $T\in L(X)$ is a {\it left essentially  Drazin invertible} operator if
 $
 a_e(T) < \infty$ and $ R(T^{a_e(T)+1})$  is closed.
 If
        $\delta_e(T) < \infty$  and $ R(T^{\delta_e(T)})$  is closed, then $T$ is called {\it right essentially  Drazin invertible}.
In the sequel   $LD^e(X)$ (resp.
$RD^e(X)$)  will denote the set of left (resp. right) essentially  Drazin invertible
 operators.

 For $T\in L(X)$, the {\it left Drazin spectrum}, the {\it right Drazin spectrum},  the {\it  Drazin spectrum}, the {\it left essentially Drazin   spectrum}, the {\it right essentially Drazin spectrum}, the {\it descent spectrum} and the {\it essential descent spectrum} are defined, respectively, by:
 \begin{eqnarray*}
      \sigma_{LD}(T) &=& \{\lambda\in\CC:T-\lambda I\notin LD(X)\}, \\
    \sigma_{RD}(T) &=& \{\lambda\in\CC:T-\lambda I\notin RD(X)\},\\
    \sigma_D(T) &=& \{\lambda\in\CC:T-\lambda I\notin D(X)\},\\
     \sigma_{LD}^e(T) &=& \{\lambda\in\CC:T-\lambda I\notin LD^e(X)\}, \\
     \sigma_{RD}^e (T)&=& \{\lambda\in\CC:T-\lambda I\notin RD^e(X)\},\\
     \sigma_{dsc} (T)&=& \{\lambda\in\CC:\delta(T-\lambda I)=\infty\},\\
     \sigma_{dsc}^e (T)&=& \{\lambda\in\CC:\delta_e(T-\lambda I)=\infty\}.
      \end{eqnarray*}
An operator $T \in
L(X)$ is said to be {\it quasi-Fredholm} if there is $d\in\NN_0$ such that $k_n(T)=0$ for  all $n\ge d$ and $R(T^{d+1})$ is closed.
The set of quasi-Fredholm operators includes many sets of operators such as left (right) Drazin invertible operators, left (right) essentially Drazin invertible operators, upper (lower) semi-B-Weyl operators (see \cite{P8}).

For $T \in
L(X)$ we say that it is {\em Kato} if $R(T)$ is closed and $N(T) \subset
R(T^n)$ for every $ n \in \mathbb{N}$. An operator  $T\in L(X)$ is  {\em nilpotent} when $T^n=0$ for some
$n \in \mathbb{N}$.
An operator $T \in L(X)$  is said to be of {\em Kato type}  if there exist closed subspaces $X_1,\ X_2$
 such that $X=X_1\oplus X_2$, $T(X_i)\subset X_i$, $i=1,2$,  $T_{|X_1}$ is nilpotent and $T_{|X_2}$ is Kato. Every operator of Kato type is a quasi-Fredholm operator.
In the case of Hilbert spaces, the set of quasi-Fredholm operators coincides
with the set of Kato type operators.

For $T \in L(X)$ and every $d\in\NN_0$, the operator range topology on $R(T^d)$ is defined by the norm $\|\cdot\|_d$  such that for every $y\in R(T^d)$,
$$
\|y\|_d=\inf\{\|x\|:x\in X,\ y=T^dx\}.
$$

Operators which have eventual topological uniform descent were introduced by Grabiner in \cite{Grabiner}:
\begin{definition} \rm
 Let $T \in L(X)$. If there is $d\in\NN_0$ for which $k_n(T)=0$ for $n\ge d$, then $T$ is said to have uniform descent for $n\ge d$. If in addition, $R(T^n)$ is closed in the operator range topology of $R(T^d)$ for $n\ge d$, then we say that $T$ has {\it eventual topological uniform descent} and, more precisely, that $T$ has {\it  topological uniform descent for} (TUD for brevity) $n\ge d$.
\end{definition}
It is easily seen that if $T$ has finite  nullity, defect, ascent or essential ascent, then it has uniform descent. If $T$ has finite descent or essential descent, then $T$ has TUD. Also, the set of operators which have TUD contains the set of quasi-Fredholm operators \cite{P8}.

For $T\in L(X)$,    the {\it Kato type  spectrum}, the {\it quasi-Fredholm  spectrum} and the {\it topological uniform descent spectrum} are defined, respectively, by:
\begin{eqnarray*}\sigma_{Kt}(T)&=&\{\lambda \in
\CC: T-\lambda\ \text{ is\ not\ of\ Kato\ type}\},\\
\sigma_{q\Phi}(T)&=&\{\lambda \in \CC: T-\lambda \ \text{ is\ not\ quasi-Fredholm}\},\\
\sigma_{TUD}(T)&=&\{\lambda \in \CC: T-\lambda \ \text{ does\ not\  have\ TUD}\}.
\end{eqnarray*}

We use the following notation (\cite{P8}, \cite{MbekhtaMuller}):

\begin{center}
\begin{tabular}{ccc} 
${\bf R_1}=\S(X)$ & ${\bf R_2}=\B_-(X)$ & ${\bf R_3}=RD(X)$ \\
${\bf R_4}=\Phi_-(X)$ & ${\bf R_5}=RD^e(X)$ &  \\
${\bf R_6}=\I(X)$ &
${\bf R_7}=\mathcal{B}_+(X)$ & ${\bf R_8}=LD(X)$  \\
${\bf R_9}= \Phi_{+}(X)$ &
${\bf R_{10}}= LD^e(X)$ &   \\
\end{tabular}
\end{center}
and
\begin{eqnarray*}
 && {\bf R_{4}^a}=\{T\in L(X) :\delta(T)<\infty\},\\&&{\bf R_{5}^a}=\{T\in L(X) :\delta^e(T)<\infty\}.
\end{eqnarray*}

For a bounded linear operator T and $n\in\NN_0$ define $T_n$ to be the
restriction of $T$ to $R(T^n)$ viewed as a map from $R(T^n)$ into $R(T^n)$ (in particular,
$T_0 = T$). If $T \in L(X) $ and if there exist an integer  $n$
 for which the range space $R(T^n)$ is closed and $T_{n}$ belongs to the class  $ {\bf R}$, we will say that $T$ belongs to the class ${\bf BR}$, where ${\bf R}\in\{{\bf R}_i:i=1,\dots,10\}\cup\{\R_4^a,\R_5^a \}\cup\{\Phi(X),\B(X), \W_+(X),\W_-(X),\W(X)\}$. For $T\in L(X)$  let $\sigma_{\R}(T)=\{\lambda\in\CC:T-\lambda I\notin \R\}$ and $\sigma_{\BB\R}(T)=\{\lambda\in\CC:T-\lambda I\notin \BB\R\}$.

  More details, if for an integer $n$ the range space $R(T^n)$ is closed and $T_n$ is   Fredholm (resp. upper semi-Fredholm, lower semi-Fredholm, Browder, upper semi-Browder, lower semi-Browder), then $T$ is called a {\it B-Fredholm}  (resp. {\it upper
semi-B-Fredholm, lower semi-B-Fredholm}, {\it B-Browder}, {\it upper
semi-B-Browder, lower semi-B-Browder} ) operator. If $T \in L(X)$ is upper or lower semi-B-Fredholm, then T is called
{\it semi-B-Fredholm}.
The index $\ind(T)$ of a semi-B-Fredholm
operator T is defined as the index of the semi-Fredholm operator $T_n$. By
\cite[Proposition 2.1]{P7} the definition of the index is independent of the integer n.
An operator $T \in L(X)$
is {\it  B-Weyl} (resp. {\it  upper semi-B-Weyl, lower semi-B-Weyl}) if $T$ is   B-Fredholm
and ind(T) = 0 (resp. $T$ is upper semi-B-Fredholm and $\ind(T) \le 0$, $T$ is lower semi-B-
Fredholm and $\ind(T) \ge 0$).

 For $T\in L(X)$, the {\it upper semi-B-Fredholm spectrum}, the {\it  lower semi-B-Fredholm spectrum}, the {\it  B-Fredholm spectrum}, the {\it  upper semi-B-Weyl spectrum}, the {\it  lower semi-B-Weyl spectrum}, the {\it  B-Weyl spectrum}, the {\it  upper semi-B-Browder spectrum}, the {\it  lower semi-B-Browder spectrum} and  the {\it  B-Browder spectrum} are defined, respectively, by:
 \begin{eqnarray*}
      \sigma_{B\Phi_+}(T) &=& \{\lambda\in\CC:T-\lambda I\ {\rm is\ not\ upper\ semi-B-Fredholm}\}, \\
    \sigma_{B\Phi_-}(T) &=& \{\lambda\in\CC:T-\lambda I\ {\rm is\ not\ lower\ semi-B-Fredholm}\},\\
    \sigma_{B\Phi}(T) &=& \{\lambda\in\CC:T-\lambda I\ {\rm is\ not\ B-Fredholm}\},\\
     \sigma_{B\W_+}(T) &=& \{\lambda\in\CC:T-\lambda I\ {\rm is\ not\ upper\ semi-B-Weyl}\}, \\
     \sigma_{B\W_-}(T) &=& \{\lambda\in\CC:T-\lambda I\ {\rm is\ not\ lower\ semi-B-Weyl}\}, \\
     \sigma_{B\W}(T) &=& \{\lambda\in\CC:T-\lambda I\ {\rm is\ not\ B-Weyl}\},\\
     \sigma_{B\B_+}(T) &=& \{\lambda\in\CC:T-\lambda I\ {\rm is\ not\ upper\ semi-B-Browder}\}, \\
     \sigma_{B\B_-}(T) &=& \{\lambda\in\CC:T-\lambda I\ {\rm is\ not\ lower\ semi-B-Browder}\}, \\
     \sigma_{B\B}(T) &=& \{\lambda\in\CC:T-\lambda I\ {\rm is\ not\ B-Browder}\}.
 \end{eqnarray*}

We recall that  the set of Drazin invertible operators (resp, $LD(X)$, $RD(X)$) coincides
with the set of B-Browder (resp.  upper
semi-B-Browder, lower semi-B-Browder) operators, while the set of left (right) essentially  Drazin invertible operator coincides
with the set of   upper (lower)
semi-B-Fredholm  operators \cite[Theorem 3.6]{P8}, \cite{aienatriolo}, \cite{aienatriolo2}. Therefore, for any $T\in L(X)$ it holds:
\begin{equation*}
  \sigma_D(T)=\sigma_{B\B}(T),\ \ \sigma_{LD}(T)=\sigma_{B\B_+}(T),\ \ \sigma_{RD}(T)=\sigma_{B\B_-}(T),
\end{equation*}
and
\begin{equation*}
  \sigma_{LD}^e(T)=\sigma_{B\Phi_+}(T),\ \ \sigma_{RD}^e(T)=\sigma_{B\Phi_-}(T).
\end{equation*}

If $K \subset \mathbb{C}$, then $\partial K$ is the
boundary of $K$, $\acc \, K$ is the set of accumulation points of
$K$,  $\inter\, K$ is the set of interior points of $K$ and $\iso\, K$ is the set of isolated points of $K$.  For a compact set $K\subset\CC$, $\eta K$ denotes its connected hull.

  The aim of  this  paper  is to give characterization of the $\BB\R$ classes through  properties such as topological uniform  descent  or quasi-Fredholmness, and   properties of the  appropriate spectra  $\sigma_{\R}$,  as well as to get some results regarding boundaries and connected hulls of  $\BB\R$-spectra.

 Q. Jiang, H. Zhong and Q. Zeng in   \cite[Theorem 3.2]{JiangJMAA} characterize the set of left  Drazin invertible operators proving that if $T-\lambda I$ has TUD, then $T-\lambda I$ is left Drazin invertible if and only if $\sigma_{ap}(T)$ does not cluster at $\lambda$, and also, if and only if $\lambda$ is not an interior point of $\sigma_{ap}(T)$.
 M. Berkani, N. Castro and S.V. Djordjevi\'c proved in \cite[Theorem 2.5]{BCD} that, under the same condition that $T-\lambda I$ has TUD,  $\sigma_{p}(T)$ does not cluster at $\lambda$ if and only if $a(T-\lambda I)<\infty$. Further
Q. Jiang, H. Zhong and Q. Zeng in   \cite[Theorem 3.4]{JiangJMAA} proved that if  $T-\lambda I$ has TUD, then $\delta(T-\lambda I)<\infty$ if and only if $\sigma_{su}(T)$ does not cluster at $\lambda$, and also, if and only if $\lambda$ is not an interior point of $\sigma_{su}(T)$.

In this paper we  characterize the sets  of upper and lower semi-B-Weyl operators, as well as the sets of left and right essentially Drazin invertible operators. We also give
further  characterisations of left and right Drazin invertible operators.
  By using Grabiner's punctured neighborhood theorem \cite[Theorem 4.7]{Grabiner}, \cite[Thorem 4.5]{P8} we prove that

  \parbox{11cm}{\begin{eqnarray*}
T\in {\bf BR}&\Longleftrightarrow&T\ {\rm is\ quasi-Fredholm}\ \ \wedge\ \  0\notin\acc\, \sigma_{{\bf R}}(T)\\&\Longleftrightarrow&T\ {\rm is\ quasi-Fredholm}\ \ \wedge\ \  0\notin\inter\, \sigma_{{\bf R}}(T),\label{senka1}
\end{eqnarray*}}  \hfill
\parbox{1cm}{\begin{eqnarray}\end{eqnarray}}

\noindent  for $\R\in\{{\bf R_2,R_3,R_4,R_5},\W_-(X)\}$.
 By an example we show that the condition that $T$ is quasi-Fredholm in the previous formulas  can not be replaced by a weaker condition that $T$ has topological uniform descent.

Further  we prove that

\parbox{11cm}{\begin{eqnarray*}
T\in {\bf BR}&\Longleftrightarrow&T\ {\rm has\ TUD}\ \ \wedge\ \  0\notin\acc\, \sigma_{{\bf R}}(T)\\&\Longleftrightarrow&T\ {\rm has\ TUD}\ \ \wedge\ \  0\notin\inter\, \sigma_{{\bf R}}(T),\label{senka2}
\end{eqnarray*}}  \hfill
\parbox{1cm}{\begin{eqnarray}\end{eqnarray}}

\noindent  for $\R\in\{{\bf  R_7,R_8,R_9,R_{10}, R^a_4, R^a_5}, \W_+(X),\W(X),\Phi(X), \B(X)\}$.

  The condition that $T$ has TUD ($T$ is quasi-Fredholm) in the previous equivalences \eqref{senka2} (\eqref{senka1}) cannot be ommited and it is demonstrated by an example.
%

As a consequence of these characterizations, for $\R\in\{{\bf R_1, R_2,R_4,R_6, R_7,R_9}\}\cup\{\W_+(X),\break\W_-(X),\W(X),\Phi(X), \B(X), L(X)^{-1}\}$  we obtain that
 $
 \inter\, \sigma_{\bf R}(T)=\inter\, \sigma_{\bf{BR}}(T)
 $,
 $
 \partial\,  \sigma_{\bf{BR}}(T)\subset \partial\,  \sigma_{\bf R}(T)
 $
 and the set
  $\sigma_{\bf R}(T)\setminus \sigma_{\bf{BR}}(T)$ consists of at most countably many isolated points.
 Also we obtain that the boundary of $\sigma_{\bf{BR}}(T)$, for
 $\R\in\{{\bf R_6, \R_7,\R_8,\R_9,\R_{10}, \R^a_4, \R^a_5},\W_+(X),\break
  \W(X),\Phi(X), \B(X)\}$
 is contained in  $\sigma_{TUD}(T)$,  while the boundary of $\sigma_{\bf{BR}}(T)$, where
  $\R\in\{{\bf R_1,R_2,R_3,R_4,R_5},\W_-(X)\}$,
  is contained in  $\sigma_{q\Phi}(T)$, and by an example it is shown that it is not  contained in the  TUD spectrum.

Boundaries of spectra originated from Fredholm theory were investigated by  K. Mili\v ci\'c and K. Veseli\'c in \cite[Theorem 7]{MV}. They  proved the following inclusions:

{\footnotesize
$$\begin{array}{ccccccccccc}
&&&&   &  &    &   & \p\sigma_{\Phi_+}(T)& & \\
&&& & & &
&\!\rotatebox{20}{$\subset$}& & \rotatebox{-20}{$\subset$}&\\
  & &  \p\sigma_{\B}(T) & \subset &
\p\sigma_{\W}(T) & \subset &  \p\sigma_{{ \Phi}}(T)&
    & &&  \p \sigma_{\Phi_{\pm}}(T).\\
&&&  & & &
&\rotatebox{-20}{$\subset$}& & \rotatebox{20}{$\subset$}&\\
& &&&   & &  & & \p\sigma_{\Phi_-}(T) & &\\
\end{array}$$
}

 V. Rako\v cevi\'c proved (see \cite[Theorem 1]{V1})
that $\partial \sigma_{\W}(T)\subset \sigma_{\W_+}(T)$ and hence there is the inclusion $\partial \sigma_{\W}(T)\subset \p\sigma_{\W_+}(T)$. In \cite[Corollary 2.5]{V2}  it is proved  that $\p\sigma_{\B}(T)\subset\p\sigma_{\B_+}(T)\subset\p\sigma_{\W_+}(T)$, as well as that $\eta\sigma_{\B}(T)= \eta\sigma_{\B_+}(T)=\eta\sigma_{\\W_+}(T)$.
  The following inclusions are known:

\bigskip

{\footnotesize
$$\begin{array}{ccccccccccc}
&&&&  \p \sigma_{\B_+}(T) & \subset &    \p\sigma_{{\W_+}}(T) &  \subset  & \p\sigma_{\Phi_+}(T)& & \\
&&& \rotatebox{20}{$\subset$} & & \rotatebox{20}{$\subset$}&
&\!\rotatebox{20}{$\subset$}& & \rotatebox{-20}{$\subset$}&\\
  & &  \p\sigma_{\B}(T) & \subset &
\p\sigma_{\W}(T) & \subset &  \p\sigma_{{ \Phi}}(T)&
  \subset  & &&  \p \sigma_{\Phi_{\pm}}(T).\\
&&& \rotatebox{-20}{$\subset$} & & \rotatebox{-20}{$\subset$}&
&\rotatebox{-20}{$\subset$}& & \rotatebox{20}{$\subset$}&\\
& &&&   \p\sigma_{\B_-}(T)& \subset& \p \sigma_{{\W_-}}(T) & \subset & \p\sigma_{\Phi_-}(T) & &\\
\end{array}$$
}

We generalize these results to the case of spectra originated from semi-B-Fredholm theory and prove  the following inclusions: 

\medskip

{\footnotesize
$$\begin{array}{ccccccccccc}
&&&&  \p \sigma_{B\B_+}(T) & \subset &    \p\sigma_{B{\W_+}}(T) &  \subset  & \p\sigma_{B\Phi_+}(T)& & \\
&&& \rotatebox{20}{$\subset$} & & \rotatebox{20}{$\subset$}&
&\!\rotatebox{20}{$\subset$}& & \rotatebox{-20}{$\subset$}&\\
  & &  \p\sigma_{B\B}(T) & \subset &
\p\sigma_{B\W}(T) & \subset &  \p\sigma_{{ B\Phi}}(T)&
  \subset  & &&  \p \sigma_{q\Phi}(T),\\
&&& \rotatebox{-20}{$\subset$} & & \rotatebox{-20}{$\subset$}&
&\rotatebox{-20}{$\subset$}& & \rotatebox{20}{$\subset$}&\\
& &&&   \p\sigma_{B\B_-}(T)& \subset& \p \sigma_{B{\W_-}}(T) & \subset & \p\sigma_{B\Phi_-}(T) & &\\
\end{array}$$
}

\bigskip

{\footnotesize
$$\begin{array}{ccccccccccc}
&&&&   \p\sigma_{B\B_+}(T) & \subset &    \p\sigma_{B{\W_+}}(T) &   && & \\
&&& \rotatebox{20}{$\subset$} & & \rotatebox{20}{$\subset$}&
&\!\rotatebox{-20}{$\subset$}& & &\\
&& \p\sigma_{B\B}(T)  & \subset & \p \sigma_{{ B\W}}(T)&
  \subset  & \p\sigma_{{ B\Phi}}(T)&\subset&\p \sigma_{B\Phi_+}(T) &\subset &\p \sigma_{TUD}(T),\\
  &&&\rotatebox{-20}{$\subset$} & & & &\rotatebox{-20}{$\subset$}& & \rotatebox{20}{$\subset$}&\\
&\ &&&  &\p\sigma_{dsc}(T) \ &   & \subset & \p\sigma_{dsc}^e(T) & &\\
\end{array}$$
}


\noindent as well as that the connected hulls of all spectra  mentioned in the previous inclusions are mutually  equal and also coincide with the connected hull of Kato type spectrum. 

As an application we get that a bounded linear operator $T$ is meromorphic, that is its non-zero spectral points are poles of its resolvent, if and only if $\sigma_{B\Phi}(T)\subset\{0\}$ and this is exactly when $\sigma_{TUD}(T)\subset\{0\}$. This result was obtained earlier  (see \cite{P13}  and \cite{Q. Jiang}).
 Q. Jiang, H. Zhong and S. Zhang in \cite[Corollary 3.3]{Q. Jiang} proved it by using the local constancy of the mappings $\lambda\mapsto K(\lambda I-T)+H_0(\lambda I-T)$ and $\lambda\mapsto \overline{K(\lambda I-T)\cap H_0(\lambda I-T)
}$ \cite[Theorem 2.6]{Q. Jiang} and results about SVEP established in \cite{JiangJMAA},
  but our method of proof is rather different and more direct.
  Q. Jiang, H. Zhong and S. Zhang   also obtained  that if $\rho_{TUD}(T)$ has only one component, then $\sigma_D(T)=\sigma_{TUD}(T)$ \cite[Theorem 3.1]{Q. Jiang} and hence, if $\sigma(T)$ is countable or contained in a line segment, then $\sigma_{D}(T)=\sigma_{TUD}(T)$ \cite[p. 1156]{Q. Jiang}. We give here an alternative proof of these results and  get more than this: if $\sigma(T)$ is  contained in a line, then
$\sigma_{D}(T)=\sigma_{TUD}(T)$, and moreover,   if $\sigma_{\bf R}(T)$ is  contained in a line for
     $\R\in\{{\bf R_6, R_7,R_8,R_9,R_{10}, R^a_4, R^a_5}, \W_+(X),\W(X),\Phi(X), \B(X)\}$, then
$\sigma_{\bf{BR}}(T)=\sigma_{TUD}(T)$. On the other side if $\R\in\{{\bf R_1,R_2,R_3,R_4}, {\bf R_5},\W_-(X)\}$ and  $\sigma_{\bf R}(T)$ is  contained in a line, then $\sigma_{\bf{BR}}(T)=\sigma_{qF}(T)$.
  We also prove that if $\sigma_{*}(T)$ is  contained in a line for  $\sigma_{*}\in\{ \sigma_{B\Phi},\sigma_{B\W}, \sigma_{LD}^e, \sigma_{B\W_+}, \sigma_{LD},   \sigma_{dsc}^e, \sigma_{dsc}  \}$,
    then $\sigma_{*}(T)=\sigma_{TUD}(T)$, while  if $\sigma_{*}(T)$ is  contained in a line for
     $\sigma_{*}\in\{  \sigma_{RD}^e, \sigma_{B\W_-}, \sigma_{RD}\}$, then $\sigma_{*}(T)=\sigma_{qF}(T)$. In particular, if   $\sigma_p(T)$ ($\sigma_{cp}(T)$) is countable or contained in a line, then  $\sigma_{LD}(T)=\sigma_{TUD}(T)$ ($\sigma_{RD}(T)=\sigma_{q\Phi}(T)$ and $\sigma_{dsc}(T)=\sigma_{TUD}(T)$).
Furthermore, by using  connected hulls we show that
  if
$\CC\setminus\sigma_{*}(T)$  has only one component where  $\sigma_{*}$ is one of  $\sigma_{q\Phi},\sigma_{Kt},\sigma_{B\Phi},\sigma_{B\W}, \sigma_{LD}^e, \sigma_{B\W_+}, \sigma_{LD}, \sigma_{RD}^e, \sigma_{B\W_-},\break \sigma_{RD}, \sigma_{dsc}^e, \sigma_{dsc}  $, then $\sigma_{*}(T)=\sigma_{D}(T)$.
  Also  we give an alternative proof of Theorem 2.10 in \cite{P13}.
        As a consequence we get that if $\sigma_{*}(T)=\partial\sigma_{*}(T)=\acc\,
\sigma_{*}(T)$, then $\sigma_{*}(T)=\sigma_{TUD}(T)$ for  $\sigma_{*}\in\{\sigma_{\W_+}, \sigma_{\W_-},\sigma_\W, \sigma_{B\W_-}, \sigma_{{\Phi_+}},  \sigma_{{\Phi_-}},  \sigma_{\Phi},\sigma_{RD}^e, \sigma_{ap}, \sigma_{su},\sigma_{\B_+}, \sigma_{\B_-}, \sigma_{\B},\sigma_{RD}, \sigma\}$. In particular, if
  $\sigma_{ap}(T)=\partial\sigma (T)$ ($\sigma_{su}(T)=\partial\sigma (T)$) and every $\lambda\in \partial\sigma (T)$ is not isolated in $\sigma (T)$,  then
$\sigma_{TUD}(T)=\sigma_{ap}(T)$ ($\sigma_{TUD}(T)=\sigma_{su}(T)$). It improves the corresponding results of P. Aiena and E. Rosas \cite[Theorem 2.10, Corollary 2.11]{aienarosas}. These results are then used to find    the TUD spectrum of arbitrary non-invertible isometry. We also  use them to find the  TUD spectrum and  B-spectra of  the forward and backward unilateral shifts on $  c_0(\NN), c(\NN), \ell_{\infty}(\NN)$ or $\ell_p(\NN)$, $p\ge 1$,  and also of $\rm Ces\acute{a}ro$ operator.

\section{Semi-B-Weyl and semi-B-Fredholm  operators}

We start with the following auxiliary assertions.
\begin{lemma}\label{B-poc}
Let    $T \in L(X)$ have TUD   for $n\ge d$ and   finite essential ascent. Then      $R(T^n)$   is closed in $X$  for   each integer $n\geq d $.
\end{lemma}
\begin{proof} Since $T$ has
finite  essential ascent and TUD   for $n\ge d$, we have  that $$\dim(N(T)\cap R(T^n))<\infty\ {\rm for\ all}\ n\ge d.$$ It means that $\alpha(T_n)<\infty$ for $T_n:R(T^n)\to R(T^n)$  and hence $\alpha(T_n^d)\le d\cdot\alpha(T_n)<\infty$. So we have that
\begin{equation}\label{o1}
  \dim(N(T^d)\cap R(T^n))<\infty\ {\rm for\ all\ }n\ge d.
\end{equation}
From \cite[Theorem 3.2]{Grabiner} it follows that  $N(T^d)  + R(T^n)$ is closed in $X$ for every $n\ge 0$. According to \eqref{o1}, $ N(T^d) \cap R(T^n)$ is closed for every $n\ge d$ and  then by \cite[Lemma 20.3]{Mu} we obtain that  $R(T^n)$ is closed for every $n\ge d$.
\end{proof}
\begin{lemma}\label{B-poc1}
 Let $T\in L(X)$. Then:

\begin{itemize}

\item[(1)]  $T$ has  TUD and $a_e(T)<\infty$  $\Longleftrightarrow$  $T$ is left essentially Drazin invertible.

\item[(2)] $T$ has TUD and $a(T)<\infty$ $\Longleftrightarrow$
  $T$ is left Drazin invertible.

\medskip

 \end{itemize}

\end{lemma}
\begin{proof}
(1) Suppose that $T$ has TUD for $n\ge  d$ and that $a_e(T)<\infty$. From Lemma \ref{B-poc} it follows that there exists $n\ge a_e(T)+1$ such that  $R(T^n)$ is closed. According to \cite[Lemma 7]{MbekhtaMuller} it follows that $R(T^{a_e(T)+1})$ is closed and hence $T$ is left essentially Drazin invertible.

The opposite inclusion is clear (see \cite[p. 166 and 172]{P8}).

(2) can be proved similarly.
\end{proof}

\medskip

 In the following two theorems we characterize upper and lower semi-B-Weyl operators.
\begin{theorem}\label{F+-} Let $\lambda\in\CC$,  $T\in L(X)$ and let $T-\lambda I$ have TUD  for $n\ge d$. Then the following statements are equivalent:
\begin{itemize}
\item[(1)] $\sigma_{\W_+}(T)$ does not cluster at $\lambda$;

\item[(2)] $\lambda$ is not an interior point of $\sigma_{\W_+}(T)$;

\item[(3)] $\sigma_{B\W_+}(T)$ does not cluster at $\lambda$;

\item[(4)] $\lambda$ is not an interior point of $\sigma_{B\W_+}(T)$;

\item[(5)] $T-\lambda I$ is an upper semi-B-Weyl operator.

\end{itemize}
\end{theorem}
\begin{proof} (1)$\Longrightarrow$(2), (3)$\Longrightarrow$(4) Clear.

(1)$\Longrightarrow$(3), (2)$\Longrightarrow$(4) It follows from the inclusion $\sigma_{B\W_+}(T)\subset \sigma_{\W_+}(T)$.

(4)$\Longrightarrow$(5) Since $T-\lambda I$ has TUD  for $n\ge d$,
from \cite[Theorem 4.7]{Grabiner} it follows that there exists an $\epsilon>0$ such that for every $\mu\in\CC$ the following implication  holds:
\begin{eqnarray}
& 0<|\mu-\lambda|<\epsilon \Longrightarrow\label{BW1}\\& c_n(T-\mu I)=c_d(T-\lambda I)\ {\rm and}\   c_n^\prime(T-\mu I)=c_d^\prime(T-\lambda I)\ {\rm for\ all\ }n\ge 0.\nonumber
\end{eqnarray}
Suppose that $\lambda$ is not an interior point of $\sigma_{B\W_+}(T)$. Then
 there  exists $\mu\in\CC$ such that $0<|\mu-\lambda|<\epsilon$ and $T-\mu I$ is an upper semi-B-Weyl operator. Therefore,
    $c_n^\prime(T-\mu I)=\dim (N(T-\mu I)\cap R((T-\mu I)^n))<\infty $ for $n$ large enough and according to \eqref{BW1} we obtain that $c_d^\prime(T-\lambda I)<\infty$,
     and so $a_e(T-\lambda I)\le d$.
    From Lemma  \ref{B-poc} it follows that 
    $R((T-\lambda I)^{d})$ and $R((T-\lambda I)^{d+1})$ are closed. As $\dim (N(T-\lambda I)\cap R((T-\lambda I)^d))=c_d^\prime(T-\lambda I)<\infty$, we have that
      the restriction of $T-\lambda I$ to $R((T-\lambda I)^{d})$ is an upper semi-Fredholm operator. Consequently,
     $T-\lambda I$ is an upper semi-B-Fredholm operator and since
    \begin{eqnarray*}
  {\rm ind }(T-\lambda I)& =&\dim (N(T-\lambda I)\cap R((T-\lambda I)^{d})-\dim  R((T-\lambda I)^{d})/  R((T-\lambda I)^{d+1})  \\
  & =&c_d^\prime(T-\lambda I)- c_d(T-\lambda I)=c_n^\prime(T-\mu I)- c_n(T-\mu I)\\&=&{\rm ind}(T-\mu I)\le 0,
\end{eqnarray*}
 it follows that $T-\lambda I$ is an upper semi-B-Weyl operator.

 (5)$\Longrightarrow$(1) Suppose that $T-\lambda I$ is an upper semi-B-Weyl operator. Then there exists $d\in\NN_0$ such that  $T-\lambda I$ has TUD  for $n\ge d$, and
  $c_d^\prime(T-\lambda I)=\dim (N(T-\lambda I)\cap R((T-\lambda I)^d))<\infty $  and $\ind(T-\lambda I)=c_d^\prime(T-\lambda I)- c_d(T-\lambda I)\le 0$. For arbitrary $\mu\in\CC$ such that $ 0<|\mu-\lambda|<\epsilon$, according to \eqref{BW1}
 we obtain that  $\alpha(T-\mu I)=c_0^\prime(T-\mu I)=c_d^\prime(T-\lambda I)<\infty$   and since $R(T-\mu I)$ is closed by \cite[Theorem 4.7]{Grabiner}, we conclude that $T-\mu I$ is upper semi-Fredholm with $\ind(T-\mu I)=c_0^\prime(T-\mu I)-c_0(T-\mu I)=c_d^\prime(T-\lambda I)- c_d(T-\lambda I)\le 0$, that is $T-\mu I$ is upper semi-Weyl. Therefore, $\lambda$ is not an accumulation point of $\sigma_{\W_+}(T)$.
\end{proof}

\begin{theorem}\label{F-+} Let $\lambda\in\CC$,  $T\in L(X)$ and let $T-\lambda I$ have TUD  for $n\ge d$. Then the following statements are equivalent:
\begin{itemize}
\item[(1)] $\sigma_{\W_-}(T)$ does not cluster at $\lambda$;

\item[(2)] $\lambda$ is not an interior point of $\sigma_{\W_-}(T)$;

\item[(3)] $\sigma_{B\W_-}(T)$ does not cluster at $\lambda$;

\item[(4)] $\lambda$ is not an interior point of $\sigma_{B\W_-}(T)$.

In particular, if $T-\lambda I$ is quasi-Fredholm, then the statements (1)-(4) are equivalent to the following satement:

\item[(5)] $T-\lambda I$ is  a lower semi-B-Weyl operator.

\end{itemize}
\end{theorem}
\begin{proof} (1)$\Longrightarrow$(2), (3)$\Longrightarrow$(4) Clear.

(1)$\Longrightarrow$(3), (2)$\Longrightarrow$(4) It follows from the inclusions $\sigma_{B\W_-}(T)\subset \sigma_{\W_-}(T)$.

(4)$\Longrightarrow$(1) 
 Suppose that $\lambda$ is not an interior point of $\sigma_{B\W_-}(T)$.
Since $T-\lambda I$ has TUD for $n\ge d$, according to \cite[Theorem 4.7]{Grabiner}  there exists
an $\epsilon>0$ such that  for every $\mu\in\CC$ the  implication \eqref{BW1} holds. From $\lambda\notin\inter\, \sigma_{B\W_-}(T)$ it follows that there exists $\mu_0\in\CC$ such that $0<|\mu_0-\lambda|<\epsilon$ i $T-\mu_0 I$ is  a lower semi-B-Weyl operator. Hence there exists $n\in\NN_0$ such that $c_n(T-\mu_0 I)=\dim R((T-\mu_0 I)^n)/R((T-\mu_0 I)^{n+1})
<\infty$ and $\ind (T-\mu_0 I)=c_n^{\prime}(T-\mu_0 I) -c_n(T-\mu_0 I)\ge 0$,  which according to \eqref{BW1} implies that $c_d(T-\lambda I)<\infty$ and $c_d^\prime(T-\lambda I)-c_d(T-\lambda I)\ge 0$. Using
\eqref{BW1} again we get that for every $\mu\in\CC$ such that $0<|\mu-\lambda|<\epsilon$ we have that $\beta(T-\mu I)=c_0(T-\mu I)=c_d(T-\lambda I)<\infty$ and hence $T-\mu I$ is lower semi-Fredholm with ${\rm ind}(T-\mu I)=c_0^{\prime}(T-\mu I)-{c_0}(T-\mu I)=c_d^\prime(T-\lambda I)-c_d(T-\lambda I)\ge 0$. This means that $\lambda$ is not an accumulated point of $\sigma_{\W_-}(T)$.

(4)$\Longrightarrow$(5) Suppose that $T-\lambda I$ is quasi-Fredholm. Then  there exists $d\in\NN_0$ such that $R(T-\lambda I)+N((T-\lambda I)^n)=R(T-\lambda I)+N((T-\lambda I)^d)$ for all $n\ge d$ and $R((T-\lambda I)^{d+1})$ is closed.
    So
        $T-\lambda I$ has TUD for $n\ge d$. From \cite[Theorem 4.7]{Grabiner} it follows that there exists an $\epsilon>0$ such that for every $\mu\in\CC$ the  implication \eqref{BW1} holds.

 Further, suppose that  $\lambda\notin \inter\, \sigma_{B\W_-}(T)$. Then there  exists $\mu\in\CC$ such that $0<|\mu-\lambda|<\epsilon$ and $T-\mu I$ is a lower semi-B-Weyl operator. Therefore,
    $c_n(T-\mu I)=\dim (R((T-\mu I)^n)/ R((T-\mu I)^{n+1})<\infty $ for $n$ large enough and according to \eqref{BW1} we obtain that $c_d(T-\lambda I)<\infty$.  As $R((T-\lambda I)^{d+1})$ is closed,
    from \cite[Lemma 12]{MbekhtaMuller},  we conclude that $R((T-\lambda I)^{d})$ is closed.  Since $\dim (R((T-\lambda I)^d)/ R((T-\lambda I)^{d+1})=c_d(T-\lambda I)<\infty$, we have that
      the restriction of $T-\lambda I$ to $R((T-\lambda I)^{d})$ is a lower semi-Fredholm operator.
     Therefore, $T-\lambda I$ is a lower semi-B-Fredholm operator and, as in the proof of the implication (4)$\Longrightarrow$(5) in Theorem \ref{F+-}, we conclude that ${\rm ind }(T-\lambda I)={\rm ind}(T-\mu I)\ge 0$. Consequently,  $T-\lambda I$ is a lower semi-B-Weyl operator.

(5)$\Longrightarrow$(1) Suppose that $T-\lambda I$ is a lower semi-B-Weyl operator. Then there is $d\in\NN_0$ such that  $T-\lambda I$ has TUD  for $n\ge d$ and hence there exists an $\epsilon>0$ such that for every $\mu\in\CC$ the  implication \eqref{BW1} holds. Also we have that
  $$c_d(T-\lambda I)=\dim (R((T-\lambda I)^d)/ R((T-\lambda I)^{d+1})<\infty $$  and $$0\le \ind(T-\lambda I)=c_d^\prime(T-\lambda I)- c_d(T-\lambda I).$$ For arbitrary $\mu\in\CC$ such that $ 0<|\mu-\lambda|<\epsilon$,
 according to \eqref{BW1}, we obtain that  $\beta(T-\mu I)=c_0(T-\mu I)=c_d(T-\lambda I)<\infty$ and  $\ind(T-\mu I)=c_0^\prime(T-\mu I)-c_0(T-\mu I)=c_d^\prime(T-\lambda I)- c_d(T-\lambda I)\ge 0$, which implies that $T-\mu I$ is a lower semi-Weyl operator. Consequently, $\lambda$ is not an accumulation point of $\sigma_{\W_-}(T)$.
\end{proof}

\begin{theorem}\label{W} Let $\lambda\in\CC$,  $T\in L(X)$ and let $T-\lambda I$ have TUD  for $n\ge d$. Then the following statements are equivalent:
\begin{itemize}
\item[(1)] $\sigma_{\W}(T)$ does not cluster at $\lambda$;

\item[(2)] $\lambda$ is not an interior point of $\sigma_{\W}(T)$;

\item[(3)]  $\sigma_{B\W}(T)$ does not cluster at $\lambda$;

\item[(4)] $\lambda$ is not an interior point of $\sigma_{B\W}(T)$;

\item[(5)] $T-\lambda I$ is a B-Weyl operator.
\end{itemize}
\end{theorem}
\begin{proof}  (4)$\Longrightarrow$(5) Suppose that $\lambda\notin \inter\, \sigma_{B\W}(T)$ and that
        $T-\lambda I$ has TUD for $n\ge d$. According to \cite[Theorem 4.7]{Grabiner}  there exists an $\epsilon>0$ such that for every $\mu\in\CC$ the  implication \eqref{BW1} holds.
   From $\lambda\notin \inter\, \sigma_{B\W}(T)$ it follows that there  exists $\mu\in\CC$ such that $0<|\mu-\lambda|<\epsilon$ and $T-\mu I$ is a B-Weyl operator. Therefore, for $n$ large enough we have that
    $c_n(T-\mu I)=\dim (R((T-\mu I)^n)/ R((T-\mu I)^{n+1}))<\infty $, $c_n^\prime(T-\mu I)=\dim (N(T-\mu I)\cap R((T-\mu I)^n))<\infty $ and $0=\ind(T-\mu I)=c_n^\prime(T-\mu I)-c_n(T-\mu I)$. According to \eqref{BW1} we obtain that $c_d(T-\lambda I)=c_d^\prime(T-\lambda I)<\infty$, that is $$\dim (N(T-\lambda I)\cap R((T-\lambda I)^{d}))=\dim ( R((T-\lambda I)^{d})/  R((T-\lambda I)^{d+1}))<\infty.$$
    It means that
      the restriction of $T-\lambda I$ to $R((T-\lambda I)^{d})$ is a Weyl operator.
     Therefore, $T-\lambda I$ is a B-Weyl operator.

The implication (5)$\Longrightarrow$(1) follows from Theorems \ref{F+-} and  \ref{F-+}.
\end{proof}

We need the following well-known results
 (see \cite{MbekhtaMuller}, \cite[Remark A (iii)]{berkani-3}, \cite[Proposition 3.1]{P7}, \cite[Corollary 1.3]{BKMO}, \cite[Corollary 2.5]{Fredj}, \cite[Theorem 4.7 and Corollary 4.8]{Grabiner}, \cite[Corollary 1.45]{Ai}).
\begin{proposition}\label{closed} For $T\in L(X)$ the set $\sigma_*(T)$ is compact, where $\sigma_*\in \{\sigma_D, \sigma_{LD}, \sigma_{LD}^e, \sigma_{B\W},\break \sigma_{B\Phi}, \sigma_{B\W_+}, \sigma_{dsc}, \sigma_{dsc}^e, \sigma_{RD}, \sigma_{RD}^e, \sigma_{B\W_-}, \sigma_{Kt}, \sigma_{q\Phi}, \sigma_{TUD}\}$.
\end{proposition}

\begin{corollary}\label{cor-W}  Let $T\in L(X)$. Then

 \begin{itemize}

 \item[(1)] \quad $\sigma_{B\W_+}(T)=\sigma_{TUD}(T)\cup\inter\, \sigma_{*}(T)=\sigma_{TUD}(T)\cup\acc\, \sigma_{*}(T)$,

\quad  for $\sigma_{*}\in\{\sigma_{\W_+}, \sigma_{B\W_+}\}$;

 \item[(2)] \quad $\sigma_{B\W_-}(T)=\sigma_{q\Phi}(T)\cup\inter\, \sigma_{*}(T)=\sigma_{q\Phi}(T)\cup\acc\, \sigma_{*}(T)$;

 \quad  for $\sigma_{*}\in\{\sigma_{\W_-}, \sigma_{B\W_-}\}$;

\item[(3)] \quad $\sigma_{B\W}(T)=\sigma_{TUD}(T)\cup\inter\, \sigma_{*}(T)=\sigma_{TUD}(T)\cup\acc\, \sigma_{*}(T)$;

 \quad  for $\sigma_{*}\in\{\sigma_{\W}, \sigma_{B\W}\}$.

\item[(4)] \quad $\inter\, \sigma_{\W_*}(T)=\inter\, \sigma_{B\W_*}(T)$,\quad for $\W_*\in\{\W_+,\W_-, \W\}$;

\item[(5)] \quad $\p\, \sigma_{B\W_*}(T)\subset \p\, \sigma_{\W_*}(T)$,\quad for $\W_*\in\{\W_+,\W_-, \W\}$;

\item[(6)] \quad $\sigma_{\W_+}(T)\setminus \sigma_{B\W_+}(T)=(\iso\, \sigma_{\W_+}(T))\setminus\sigma_{TUD}(T)$,

\quad $\sigma_{\W_-}(T)\setminus \sigma_{B\W_-}(T)=(\iso\, \sigma_{\W_-}(T))\setminus\sigma_{qF}(T)$,

\quad $\sigma_{\W}(T)\setminus \sigma_{B\W}(T)=(\iso\, \sigma_{\W}(T))\setminus\sigma_{TUD}(T)$;

  \item[(7)] \quad  $\sigma_{\W_*}(T)\setminus \sigma_{B\W_*}(T)$, where $\W_*\in\{\W_+,\W_-, \W\}$, consists of at most countably many isolated points.

\end{itemize}
\end{corollary}
\begin{proof} (1) Let $\sigma_*\in\{\sigma_{\W_+}, \sigma_{B\W_+}\}$.  From Theorem \ref{F+-} it follows that $
T-\lambda I\ {\rm is\ upper\ semi-Weyl}$  if and only if $
T-\lambda I\ {\rm has\  TUD}$ and $\lambda$ is not an interior point of $\sigma_*(T)$, that is there is  the following equality:
\begin{equation}\label{zamena}
  \sigma_{B\W_+}(T)=\sigma_{TUD}(T)\cup\inter\, \sigma_*(T).
\end{equation}

Also from Theorem \ref{F+-} it follows that $T-\lambda I\ {\rm is\ upper\ semi-Weyl}$ if and only if $
T-\lambda I\ {\rm has\  TUD}$ and $\lambda$ is not an accumulation  point of $\sigma_*(T)$,
which implies that $\sigma_{B\W_+}(T)=\sigma_{TUD}(T)\cup\acc\, \sigma_*(T)$.

The  equalities in (2) and (3) follow from Theorems    \ref{F-+} and \ref{W}, respectively.

(4) For $\W_*\in\{\W_+,\W_-, \W\}$, from (1), (2) and (3)  it follows that $\inter \, \sigma_{\W_*}(T)\subset \sigma_{B\W_*}(T)$ and hence, $\inter \, \sigma_{\W_*}(T)\subset \inter\, \sigma_{B\W_*}(T)$. The converse  inclusion follows from the inclusion $\sigma_{B\W_*}(T)\subset \sigma_{\W_*}(T)$.

(5) Since $\sigma_{B\W_*}(T)$ is closed (Proposition \ref{closed}), we have that $\p\, \sigma_{B\W_*}(T)\subset \sigma_{B\W_*}(T)$. As $\sigma_{B\W_*}(T)\subset\sigma_{\W_*}(T)$ and $\sigma_{\W_*}(T)=\p\, \sigma_{\W_*}(T)\cup\inter\, \sigma_{\W_*}(T)$ since $\sigma_{\W_*}(T)$ is also closed,  from (4) it follows that  $\p\, \sigma_{B\W_*}(T)\subset  \p\, \sigma_{\W_*}(T)$.

(6) Let $\lambda\in \sigma_{\W_+}(T)\setminus \sigma_{B\W_+}(T)$. From  (1) we get that $\lambda\notin\acc\, \sigma_{\W_+}(T)$ and hence, $\lambda\in\iso \, \sigma_{\W_+}(T)$. As $\lambda\notin \sigma_{B\W_+}(T)$, it follows that $\lambda\notin \sigma_{TUD}(T)$ and so, $\lambda\in (\iso\, \sigma_{\W_+}(T))\setminus\sigma_{TUD}(T)$.

Suppose that $\lambda\in(\iso\, \sigma_{\W_+}(T))\setminus\sigma_{TUD}(T)$. Then $\lambda\in \sigma_{\W_+}(T)$, $\lambda\notin\acc \, \sigma_{\W_+}(T)$ and $T-\lambda I$ has TUD. According to Theorem \ref{F+-}
we get that $T-\lambda I$ is upper semi-B-Weyl and thus, $\lambda\in\sigma_{\W_+}(T)\setminus \sigma_{B\W_+}(T)$.

The rest of equalities can be proved similarly.

(7) follows from (6).
\end{proof}

In the following theorem we  characterize left essentially Drazin invertible operators, that is,  upper semi-B-Fredholm operators.

\begin{theorem}\label{SF+} Let $\lambda\in\CC$,  $T\in L(X)$ and let $T-\lambda I$ have TUD  for $n\ge d$. Then the following statements are equivalent:

\begin{itemize}

\item[(1)] $\sigma_{\Phi_+}(T)$ does not cluster at $\lambda$;

\item[(2)] $\lambda$ is not an interior point of $\sigma_{\Phi_+}(T)$;

\item[(3)] $\sigma_{LD}^e(T)$ does not cluster at $\lambda$;

\item[(4)] $\lambda$ is not an interior point of $\sigma_{LD}^e(T)$;

\item[(5)] $a_e(T-\lambda I)<\infty$;

\item[(6)] $T-\lambda I$ is left essentially Drazin invertible.

\end{itemize}
\end{theorem}
\begin{proof}

(1)$\Longrightarrow$(2), (3)$\Longrightarrow$(4) Obvious.

(1)$\Longrightarrow$(3), (2)$\Longrightarrow$(4) It follows from the inclusions $\sigma_{LD}^e(T)\subset \sigma_{\Phi_+}(T)$.

(4)$\Longrightarrow$(5) Suppose that $\lambda$ is not an interior point of $\sigma_{LD}^e(T)$.  Since $T-\lambda I$ has TUD for $n\ge d$, according to \cite[Theorem 4.7]{Grabiner}  there exists an $\epsilon>0$ such that if $0<|\lambda-\mu|<\epsilon$ we have that
\begin{equation}\label{fgf}
  c_n^\prime(T-\mu I)=c_d^\prime(T-\lambda I)\ {\rm for\ all\ }n\ge 0.
\end{equation}
  Since $\lambda\notin \inter\, \sigma_{LD}^e(T)$, there is $\mu\in \CC$ such that $0<|\mu-\lambda|<\epsilon$ and $T-\mu I$ is left essentially Drazin invertible. Thus  $a_e(T-\mu I)<\infty$, which implies that $c_n^\prime(T-\mu I)<\infty$ for some $n\in\NN_0$. According to \eqref{fgf} we conclude that $c_d^\prime(T-\lambda I)<\infty$ and hence $a_e(T-\lambda I)\le d$.

(5)$\Longrightarrow$(6) It follows from Lemma  \ref{B-poc1} (1).

(6)$\Longrightarrow$(5) It is obvious.

(5)$\Longrightarrow$(1) Let $a_e(T-\lambda I)<\infty$. Since $T-\lambda I$ has TUD, from  \cite[Corolary 4.8 (f)]{Grabiner} we get that there is an $\epsilon>0$ such that for every   $\mu\in\CC$, from $0<|\lambda-\mu|<\epsilon$ it follows that $T-\mu I$ is upper semi-Fredholm. This means that $\lambda$ is not an accumulation points of  $\sigma_{\Phi_+}(T)$.
\end{proof}
%
%
%
%
%
%
%
%
%
%

We need the following result.
\begin{proposition}\label{prJ+}\cite[Proposition 3.4]{P8}  Let $T\in L(X)$. Then

\begin{itemize}
\item[(1)]
 $T$ is quasi-Fredholm and $\delta(T)<\infty$  $\Longleftrightarrow$  $T$ is right Drazin invertible.

\item[(2)] $T$ is quasi-Fredholm and $\delta_e(T)<\infty$ $\Longleftrightarrow$ $T$ is right essentially Drazin invertible.

\end{itemize}
\end{proposition}

%

%

\bex \label{primer} Let $H$ be a
Hilbert space with an orthonormal basis
$\{e_{ij}\}^{\infty}_{i,j=1}$ and let the operator T defined by:\\
 $$
 Te_{i,j}=
\left\{
\begin{tabular}{ll}
$ 0 $  if  $ j = 1   ,$\\
 $\frac{1}{i}e_{i,1}$, if  $j=2$\\
$ e_{i, j-1},$  otherwise \\
\end{tabular}
\right.
$$

 It is easily seen   that  $R(T) =
R(T^{2})$  and $R(T)$ is not closed. Hence $R(T^n)$ is not closed
for  all $ n\geq 1$ and so  $T$ is neither a right Drazin invertible operator nor a right essentially Drazin invertible operator.
However, since $R(T) = R(T^2),$ then $T$ has uniform
descent for $n\geq 1 $ and $N(T) + R(T) = X $. Hence  $N(T) +
R(T)$ is closed and from \cite[Theorem 3.2]{Grabiner}  it follows that
$ T$ has TUD for $ n \geq 1.$ We remark that   finite descent   or  finite essential descent of a bounded  operator imply that it has TUD but does not imply closeness of ranges of its powers.  So, $T$ is an operator with $\delta(T)=\delta_e(T)<\infty$ which hence has TUD, but $T$ is neither right Drazin invertible nor right essentially Drazin invertible and this  shows that the condition that $T$ is quasi-Fredholm in the assertions (1) and (2) in Proposition \ref{prJ+} can neither be omitted nor
replaced by a weaker condition that $T$  has TUD.

\eex

In the following theorem we give some characterizations of right essentially Drazin invertible, that is, lower semi-B-Fredholm operators.

\begin{theorem}\label{SF-} Let $\lambda\in\CC$,  $T\in L(X)$ and let $T-\lambda I$ have TUD  for $n\ge d$. Then the following statements are equivalent:

\begin{itemize}

\item[(1)] $\sigma_{\Phi_-}(T)$ does not cluster at $\lambda$;

\item[(2)] $\lambda$ is not an interior point of $\sigma_{\Phi_-}(T)$;

\item[(3)] $\sigma_{RD}^e(T)$ does not cluster at $\lambda$;

\item[(4)] $\lambda$ is not an interior point of $\sigma_{RD}^e(T)$;

\item[(5)] $\sigma_{dsc}^e(T)$ does not cluster at $\lambda$;

\item[(6)] $\lambda$ is not an interior point of $\sigma_{dsc}^e(T)$;

\item[(7)] $\delta_e(T-\lambda I)<\infty$.

In particular, if $T-\lambda I$ is quasi-Fredholm then the statements (1)-(7) are equivalent to the following satement:

\item[(8)] $T-\lambda I$ is right essentially Drazin invertible.


\end{itemize}
\end{theorem}
\begin{proof}
(1)$\Longrightarrow$(2), (3)$\Longrightarrow$(4), (5)$\Longrightarrow$(6) Obvious.

(1)$\Longrightarrow$(3)$\Longrightarrow$(5), (2)$\Longrightarrow$(4)$\Longrightarrow$(6) It follows from the inclusions $\sigma_{dsc}^e(T)\subset\sigma_{RD}^e(T)\subset \sigma_{\Phi_-}(T)$.

(6)$\Longrightarrow$(7) Suppose that $\lambda$ is not an interior point of $\sigma_{dsc}^e(T)$. Since $T-\lambda I$ has TUD  for $n\ge d$, by  \cite[Theorem 4.7]{Grabiner}  there exists an $\epsilon>0$ such that for every   $\mu\in\CC$, from $0<|\lambda-\mu|<\epsilon$ it follows that  $c_n(T-\mu I)=c_d(T-\lambda I)$ for all $n\ge 0$. Since $\lambda\notin \inter\, \sigma_{dsc}^e(T)$, there is $\mu\in \CC$ such that $0<|\mu-\lambda|<\epsilon$ and $\delta_e(T-\mu I)<\infty$. This implies that $c_d(T-\lambda I)<\infty$ and hence $\delta_e(T-\lambda I)\le d$.

(7)$\Longrightarrow$(1) Let $\delta_e(T-\lambda I)<\infty$. Then $T-\lambda I$ has TUD and from  \cite[Corolary 4.8 (g)]{Grabiner} it follows that there is an $\epsilon>0$ such that if $0<|\lambda-\mu|<\epsilon$ we have that $T-\mu I$ is lower semi-Fredholm. This means that $\lambda$ is not an accumulation points of  $\sigma_{\Phi_-}(T)$.

Under assumption that $T-\lambda I$ is quasi-Fredholm, the equivalence  (7)$\Longleftrightarrow$(8)  follows from  Proposition \ref{prJ+} (2).
\end{proof}


%
%
%
%
%
%

\begin{theorem}\label{F} Let $\lambda\in\CC$,  $T\in L(X)$ and let $T-\lambda I$ have TUD  for $n\ge d$. Then the following statements are equivalent:
\begin{itemize}
\item[(1)] $\sigma_{\Phi}(T)$ does not cluster at $\lambda$;

\item[(2)] $\lambda$ is not an interior point of $\sigma_{\Phi}(T)$;

\item[(3)] $\sigma_{B\Phi}(T)$ does not cluster at $\lambda$;

\item[(4)] $\lambda$ is not an interior point of $\sigma_{B\Phi}(T)$;

\item[(5)] $T-\lambda I$ is a B-Fredholm operator.

\end{itemize}
\end{theorem}
\begin{proof} (4)$\Longrightarrow$(5): It can be proved similarly to the proof of the implication (4)$\Longrightarrow$(5) in Theorem \ref{W}.

(5)$\Longrightarrow$(1)
It follows from Theorems \ref{SF+} and  \ref{SF-}.
\end{proof}
\begin{corollary}\label{cor-F}  Let $T\in L(X)$. Then

 \begin{itemize}

 \item[(1)] \quad $\sigma_{LD}^e(T)=\sigma_{TUD}(T)\cup\inter\, \sigma_{*}(T)=\sigma_{TUD}(T)\cup\acc\, \sigma_{*}(T)$,

\quad  where $\sigma_{*}\in\{\sigma_{\Phi_+}, \sigma_{LD}^e\}$;


 \item[(2)] \quad $\sigma_{dsc}^e(T)=\sigma_{TUD}(T)\cup\inter\, \sigma_{*}(T)=\sigma_{TUD}(T)\cup\acc\, \sigma_{*}(T),$

 \quad  where $\sigma_{*}\in\{\sigma_{\Phi_-}, \sigma_{RD}^e, \sigma_{dsc}^e\}$;

 \item[(3)] \quad $\sigma_{RD}^e(T)=\sigma_{q\Phi}(T)\cup\inter\, \sigma_{*}(T)=\sigma_{q\Phi}(T)\cup\acc\, \sigma_{*}(T),$

 \quad where $\sigma_{*}\in\{\sigma_{\Phi_-}, \sigma_{RD}^e, \sigma_{dsc}^e\}$;

\item[(4)] \quad $\sigma_{B\Phi}(T)=\sigma_{TUD}(T)\cup\inter\, \sigma_{*}(T)=\sigma_{TUD}(T)\cup\acc\, \sigma_{*}(T)$,

\quad  where $\sigma_{*}\in\{\sigma_{\Phi}, \sigma_{B\Phi}\}$;

\item[(5)] \quad $\inter\, \sigma_{\Phi_+}(T)=\inter\, \sigma_{LD}^e(T)$,

\quad $\inter\, \sigma_{\Phi_-}(T)=\inter\, \sigma_{RD}^e(T)=\inter\, \sigma_{dsc}^e(T)$,

\quad $\inter\, \sigma_{\Phi}(T)=\inter\, \sigma_{B\Phi}(T)$;

\item[(6)] \quad $\p\, \sigma_{LD}^e(T)\subset \p \, \sigma_{\Phi_+}(T)$,

\quad $\p\, \sigma_{dsc}^e(T)\subset \p\, \sigma_{RD}^e(T)\subset\p\, \sigma_{\Phi_-}(T)$,

\quad $\p\, \sigma_{B\Phi}(T)\subset\p\, \sigma_{\Phi}(T)$;

\item[(7)] \quad $\sigma_{\Phi_+}(T)\setminus \sigma_{LD}^e(T)=(\iso\, \sigma_{\Phi_+}(T))\setminus\sigma_{TUD}(T)$,

\quad $\sigma_{\Phi_-}(T)\setminus \sigma_{dsc}^e(T)=(\iso\, \sigma_{\Phi_-}(T))\setminus\sigma_{TUD}(T)$,

\quad $\sigma_{\Phi_-}(T)\setminus \sigma_{RD}^e(T)=(\iso\, \sigma_{\Phi_-}(T))\setminus\sigma_{qF}(T)$,

\quad $\sigma_{\Phi}(T)\setminus \sigma_{B\Phi}(T)=(\iso\, \sigma_{\Phi}(T))\setminus\sigma_{TUD}(T)$;

  \item[(8)] \quad $\sigma_{\Phi_+}(T)\setminus \sigma_{LD}^e(T)$, $\sigma_{\Phi_-}(T)\setminus \sigma_{dsc}^e(T)$, $\sigma_{\Phi_-}(T)\setminus \sigma_{RD}^e(T)$,  $\sigma_{\Phi}(T)\setminus \sigma_{B\Phi}(T)$ are at most countable.

\end{itemize}
\end{corollary}
\begin{proof} (1) follows from  Theorem \ref{SF+}, (2) and (3) follow from Theorem \ref{SF-} and (4) follows from Theorem  \ref{F}. (5) and (7) follow from (1), (2), (3) and (4), while  (6) follows from (5) and Proposition \ref{closed}. (8) follows from (7).
\end{proof}

Further   we focus to left and right Drazin invertible operators. Q. Jiang, H. Zhong and Q. Zeng     proved that if $\lambda\in\CC$,  $T\in L(X)$ and $T-\lambda I$ has TUD  for $n\ge d$, then the following statements are equivalent (see \cite[Theorem 3.2]{JiangJMAA} and the proof of this theorem):
\begin{itemize}

\item[(1)] $T-\lambda I$ is left Drazin invertible;

\item[(2)] $a(T-\lambda I)<\infty$;

\item[(3)] $\sigma_{ap}(T)$ does not cluster at $\lambda$;

\item[(4)] $\lambda$ is not an interior point of $\sigma_{ap}(T)$,


\end{itemize}

\noindent while M. Berkani, N. Castro and S.V. Djordjevi\'c  proved in \cite[Theorem 2.5]{BCD} that, under the same condition that $T-\lambda I$ has TUD,   $\sigma_{p}(T)$ does not cluster at $\lambda$ if and only if $a(T-\lambda I)<\infty$. In the following theorem we add some characterisations of left Drazin invertible operators.

\begin{theorem}\label{a} Let $\lambda\in\CC$,  $T\in L(X)$ and let $T-\lambda I$ have TUD  for $n\ge d$. Then the following statements are equivalent:
\begin{itemize}

\item[(1)] $\lambda$ is not an interior point of $\sigma_{p}(T)$;

\item[(2)] $\sigma_{{\cal B}_+}(T)$ does not cluster at $\lambda$;

\item[(3)] $\lambda$ is not an interior point of $\sigma_{{\cal B}_+}(T)$;

\item[(4)] $\sigma_{LD}(T)$ does not cluster at $\lambda$;

\item[(5)] $\lambda$ is not an interior point of $\sigma_{LD}(T)$;


\item[(6)] $T-\lambda I$ is left Drazin invertible.

\end{itemize}
\end{theorem}
\begin{proof}

 (2)$\Longrightarrow$(3), (4)$\Longrightarrow$(5) It is obvious.

(2)$\Longrightarrow$(4),  (3)$\Longrightarrow$(5) It follows from the inclusion      $\sigma_{LD}(T)\subset\sigma_{{\cal B}_+}(T)$.

(1)$\Longrightarrow$(6) Suppose that $\lambda$ is not an interior point of $\sigma_{p}(T)$. Since $T-\lambda I$ has TUD, from \cite[Corolary 4.8 (d)]{Grabiner} it follows that $a=a(T-\lambda I)<\infty$. Now from Lemma \ref{B-poc1} (2) we get that $T-\lambda I$ is a left Drazin invertible operator.

(5)$\Longrightarrow$(6) Suppose that $\lambda$ is not an interior point of $\sigma_{LD}(T)$. As $T-\lambda I$ has TUD, according to \cite[Corolary 4.8, (a)]{Grabiner} we conclude that $a(T-\lambda I)<\infty$ and by  Lemma \ref{B-poc1} (2) $T-\lambda I$ is left Drazin invertible.


(6)$\Longrightarrow$(2) 
It follows from the implication (6)$\Longrightarrow$(5) in \cite[Theorem 3.2]{JiangJMAA}.

(6)$\Longrightarrow$(1) It follows from \cite[Theorem 2.5]{BCD}.
\end{proof}


Q. Jiang, H. Zhong and Q. Zeng in   \cite[Theorem 3.4]{JiangJMAA} proved that if $\lambda\in\CC$,  $T\in L(X)$ and $T-\lambda I$ has TUD  for $n\ge d$, then the following statements are equivalent:
\begin{itemize}
\item[(1)] $\sigma_{su}(T)$ does not cluster at $\lambda$;

\item[(2)] $\lambda$ is not an interior point of $\sigma_{su}(T)$;

\item[(3)] $\delta(T-\lambda I)<\infty$.

\end{itemize}

In the following  theorem we add some  statements equivalent to those ones  in \cite[Theorem 3.4]{JiangJMAA}.

\begin{theorem}\label{su} Let $\lambda\in\CC$,  $T\in L(X)$ and let $T-\lambda I$ have TUD  for $n\ge d$. Then the following statements are equivalent:
\begin{itemize}
%

\item[(1)] $\sigma_{cp}(T)$ does not cluster at $\lambda$;

\item[(2)] $\lambda$ is not an interior point of $\sigma_{cp}(T)$;

\item[(3)] $\sigma_{{\cal B}_-}(T)$ does not cluster at $\lambda$;

\item[(4)] $\lambda$ is not an interior point of $\sigma_{{\cal B}_-}(T)$;

\item[(5)] $\sigma_{RD}(T)$ does not cluster at $\lambda$;

\item[(6)] $\lambda$ is not an interior point of $\sigma_{RD}(T)$;

\item[(7)] $\sigma_{dsc}(T)$ does not cluster at $\lambda$;

\item[(8)] $\lambda$ is not an interior point of $\sigma_{dsc}(T)$;

\item[(9)] $\delta(T-\lambda I)<\infty$.

In particular, if $T-\lambda I$ is quasi-Fredholm, then the statements (1)-(7) are equivalent to the following satement:

\item[(10)] $T-\lambda I$ is right Drazin invertible.

\end{itemize}
\end{theorem}
\begin{proof}
 (1)$\Longrightarrow$(2), (3)$\Longrightarrow$(4), (5)$\Longrightarrow$(6), (7)$\Longrightarrow$(8) Obvious.

  (3)$\Longrightarrow$(5)$\Longrightarrow$(7), (4)$\Longrightarrow$(6)$\Longrightarrow$(8) It follows from the inclusions $\sigma_{dsc}(T)\subset \sigma_{RD}(T)\subset\sigma_{{\cal B}_-}(T)$.

(2)$\Longrightarrow$(1), (2)$\Longrightarrow$(3) Suppose that $\lambda$ is not an interior point of $\sigma_{cp}(T)$. Since $T-\lambda I$ has TUD for $n\ge d$, from \cite[Theorem 4.7]{Grabiner} we have that  there is an $\epsilon>0$ such that if  $0<|\lambda-\mu|<\epsilon$ it follows that $R(T-\mu I)$ is closed and
\begin{equation}\label{kuc}
 c_n(T-\mu I)=c_d(T-\lambda I)\ {\rm for\ all\ } n\in\NN_0.
\end{equation}
 From $\lambda\notin \inter\, \sigma_{cp}(T)$ it follows that there exists $\mu_0\in\CC$ such that $0<|\lambda-\mu_0|<\epsilon$ and $T-\mu_0 I$ has dense range. As $R(T-\mu_0 I)$ is closed, it implies that $T-\mu_0 I$ is onto and hence $c_n(T-\mu_0 I)=0$ for all $n\in\NN_0$. Consequently, $c_d(T-\lambda I)=0$ and hence for all $\mu\in\CC$ such that  $0<|\lambda-\mu|<\epsilon$ we have that $\beta(T-\mu I)=c_0(T-\mu I)=0$, i.e. $T-\mu I$ is surjective,  which means that  $\lambda\notin\acc\, \sigma_{cp}(T)$ and $\lambda\notin\acc\, \sigma_{{\cal B}_-}(T)$.

(8)$\Longrightarrow$(9) Suppose that $\lambda$ is not an interior point of $\sigma_{dsc}(T)$. Since $T-\lambda I$ has TUD for $n\ge d$, from \cite[Theorem 4.7]{Grabiner} we have that  there is an $\epsilon>0$ such that if  $0<|\lambda-\mu|<\epsilon$, then the equalities \eqref{kuc} hold.  From $\lambda\notin \inter\, \sigma_{dsc}(T)$ we have that there exists $\mu_0\in\CC$ such that $0<|\lambda-\mu_0|<\epsilon$ and $\delta(T-\mu_0 I)<\infty$. So there is $n\in\NN_0$ such that $c_n(T-\mu_0 I)=0$ and hence, according to \eqref{kuc}, it follows that $c_d(T-\lambda I)=0$. Thus $\delta(T-\lambda I)<\infty$.

(9)$\Longrightarrow$(1) It follows from  \cite[Corollary 4.8 (c)]{Grabiner}.

Under assumption that $T-\lambda I$ is quasi-Fredholm, the equivalence  (9)$\Longleftrightarrow$(10)  follows from  Proposition \ref{prJ+} (1).
\end{proof}

\begin{remark}\label{remark1} \rm Since  the operator $T$ in Example \ref{primer} has the finite descent, then according to \cite[Theorem 4.7 and Corollary 4.8]{Grabiner}  there exists an $\epsilon>0$ such that for $\mu\in\CC$ from $0<|\mu|<\epsilon$ it follows that $\delta(T-\mu I)=0$, i.e. $T-\mu I$ is surjective. This means that $0$ is not an accumulation point of $\sigma_{su}(T)$, as well as $\sigma_{\Phi_-}(T)$, $\sigma_{\W_-}(T)$, $\sigma_{RD}(T)$, $\sigma_{B\W_-}(T)$ and $\sigma_{RD}^e(T)$. As for every $n\in\NN$, $R(T^n)=R(T)$ is not closed, then $T$ is neither a lower semi-Fredholm nor  a lower semi-B Weyl operator, and as we have already mentioned $T$ is neither right Drazin invertible nor right essentially Drazin invertible. This means that the condition that $T-\lambda I$ is quasi-Fredholm in Theorems \ref{F-+}, \ref{SF-} and \ref{su}   can not be replaced by a weaker condition that $T-\lambda I$ has TUD.
\end{remark}

The next theorem follows immediately from  \cite[Theorems 3.2 and  3.4]{JiangJMAA} and Theorems \ref{a} and \ref{su}.
\begin{theorem}\label{Dra}
Let $\lambda\in\CC$,  $T\in L(X)$ and let $T-\lambda I$ have TUD  for $n\ge d$. Then the following statements are equivalent:
\begin{itemize}
\item[(1)] $\sigma(T)$ does not cluster at $\lambda$;

\item[(2)] $\lambda$ is not an interior point of $\sigma
(T)$;

\item[(3)] $\sigma_{\B}(T)$ does not cluster at $\lambda$;

\item[(4)] $\lambda$ is not an interior point of $\sigma_{\B}
(T)$;

\item[(5)] $\sigma_{D}(T)$ does not cluster at $\lambda$;

\item[(6)] $\lambda$ is not an interior point of $\sigma_{D}(T)$;

\item[(7)] $T-\lambda I$ is Drazin invertible.
\end{itemize}
\end{theorem}

%

\begin{corollary}\label{cor-Dra}  Let $T\in L(X)$. Then
\begin{itemize}
 \item[(1)]
\quad $\sigma_{LD}(T)=\sigma_{TUD}(T)\cup\inter\, \sigma_*(T)=\sigma_{TUD}(T)\cup\acc\, \sigma_*(T),$

 \quad where $\sigma_*\in\{\sigma_p, \sigma_{ap}, \sigma_{{\cal B}_+},\sigma_{LD}\}$;

 \item[(2)] \quad $\sigma_{dsc}(T)=\sigma_{TUD}(T)\cup\, \inter\, \sigma_{*}(T)=\sigma_{TUD}(T)\cup\, \acc\, \sigma_{*}(T),$

\quad where $\sigma_{*}\in\{\sigma_{su},\sigma_{cp}, \sigma_{\B_-}, \sigma_{RD}, \sigma_{dsc}\}$;

 \item[(3)] \quad $\sigma_{RD}(T)=\sigma_{q\Phi}(T)\cup\, \inter\, \sigma_{*}(T)=\sigma_{q\Phi}(T)\cup\, \acc\, \sigma_{*}(T),$

\quad where $\sigma_{*}\in\{\sigma_{su},\sigma_{cp}, \sigma_{\B_-}, \sigma_{RD}, \sigma_{dsc}\}$;

\item[(4)] \quad $\sigma_D(T)=\sigma_{TUD}(T)\cup\inter\, \sigma_*(T)=\sigma_{TUD}(T)\cup\acc\, \sigma_*(T),$

\quad where $\sigma_*\in\{\sigma,  \sigma_{\B}, \sigma_D\}$;

\item[(5)] \quad $\inter\, \sigma_{ap}(T)=\inter\, \sigma_{\B_+}(T)=\inter\, \sigma_{LD}(T)$,

\quad $\inter\, \sigma_{su}(T)=\inter\, \sigma_{\B_-}(T)=\inter\, \sigma_{RD}(T)=\inter\, \sigma_{dsc}(T)$,

\quad $\inter\, \sigma(T)=\inter\, \sigma_{\B}(T)=\inter\, \sigma_{D}(T)$;

\item[(6)] \quad $\p\, \sigma_{LD}(T)\subset\p\, \sigma_{\B_+}(T)\subset\p\, \sigma_{ap}(T)$,

\quad $\p\, \sigma_{dsc}(T)\subset\p\, \sigma_{RD}(T)\subset\p\, \sigma_{\B_-}(T)\subset\p\, \sigma_{su}(T)$,

\quad $\p\, \sigma_{D}(T)\subset\p\, \sigma_{\B}(T)\subset\p\, \sigma(T)$;

\item[(7)] \quad $\sigma_{*}(T)\setminus \sigma_{LD}(T)=(\iso\, \sigma_{*}(T))\setminus\sigma_{TUD}(T)$ for $\sigma_{*}\in\{\sigma_{ap}, \sigma_{\B_+}\}$,

\quad $\sigma_{*}(T)\setminus \sigma_{dsc}(T)=(\iso\, \sigma_{*}(T))\setminus\sigma_{TUD}(T)$ for  $\sigma_{*}\in\{\sigma_{su}, \sigma_{\B_-},\sigma_{RD} \}$,

\quad $\sigma_{*}(T)\setminus \sigma_{RD}(T)=(\iso\, \sigma_{*}(T))\setminus\sigma_{qF}(T)$  for  $\sigma_{*}\in\{\sigma_{su}, \sigma_{\B_-},\sigma_{dsc} \}$,

\quad $\sigma_{*}(T)\setminus \sigma_{D}(T)=(\iso\, \sigma_{*}(T))\setminus\sigma_{TUD}(T)$  for  $\sigma_{*}\in\{\sigma, \sigma_{\B} \}$.

\end{itemize}
\end{corollary}
\begin{proof} It follows from Theorems \ref{a} and \ref{su}, \cite[Theorem 3.2]{JiangJMAA}, \cite[Theorem 2.5]{BCD},
 Theorem \ref{Dra} and Proposition \ref{closed}, similarly to the proof of Corollary \ref{cor-W}.
%
%
%
\end{proof}


We remark that from \cite[Lemma 3.1]{P8} it follows that
\begin{equation}\label{lk1}
  {\bf BR^a_4=R^a_4},\ \ {\bf BR^a_5=R^a_5}.
\end{equation}

Now we can formulate  a general assertion:
\begin{theorem}\label{final} Let $T\in L(X)$.

\begin{itemize}
\item[(1)]
If $\R\in\{{\bf R_1,R_2,R_3,R_4,R_5},\W_-(X)\}$, then
\begin{eqnarray*}
  T\in {\bf BR}&\Longleftrightarrow& T\ {\rm is\ quasi-Fredholm}\ \ \wedge\ \ 0\notin\acc\, \sigma_{{\bf R}}(T)\\
    &\Longleftrightarrow& T\ {\rm is\ quasi-Fredholm}\ \ \wedge\ \  0\notin\inter\, \sigma_{{\bf R}}(T).
 \end{eqnarray*}

 \bigskip

If $\R\in\{{\bf R_6, R_7,R_8,R_9,R_{10}, R^a_4, R^a_5}, \W_+(X),\W(X),\Phi(X), \B(X), L(X)^{-1}\}$, then
\begin{eqnarray*}
  T\in {\bf BR}&\Longleftrightarrow& T\ {\rm has\ TUD}\ \ \wedge\ \  0\notin\acc\, \sigma_{{\bf R}}(T)\\
    &\Longleftrightarrow& T\ {\rm has\ TUD}\ \ \wedge\ \  0\notin\inter\, \sigma_{{\bf R}}(T).
 \end{eqnarray*}

 \item[(2)] If $\R\in\{{\bf R_1, R_2,R_4,R_6, R_7,R_9}\}\cup\{\W_+(X),\W_-(X),\W(X),\Phi(X), \B(X), L(X)^{-1}\}$, then
 \begin{eqnarray*}
   \inter\, \sigma_{\bf R}(T) &=& \inter\, \sigma_{\bf{BR}}(T), \\
   \p \, \sigma_{\bf{BR}}(T)&\subset&\p\, \sigma_{\bf R}(T)
 \end{eqnarray*}
  and
  $\sigma_{\bf R}(T)\setminus \sigma_{\bf{BR}}(T)$ consists of at most countably many isolated points.
 \end{itemize}
 \end{theorem}

\section{Boundaries and connected hulls of corresponding\\ spectra}

The {\it connected hull}  of a compact subset $K$ of the complex
plane $\CC$, denoted by $\eta K$, is the complement of the unbounded
component of $\CC\setminus K$ \cite[Definition
7.10.1]{H}. Given a compact subset $K$ of the plane, a hole of $K$
is a bounded component of $\CC\setminus K$, and so a hole of $K$ is
a component of $\eta K\setminus K$.

We shall need the following well-known result (see \cite[Theorem 1.2, Theorem 1.3]{HW}, \cite[Theorem 7.10.2, and
7.10.3]{H}).

\begin{proposition}\label{pocetna}
Let $K,H\subset \CC$ be compact and let
$$
\partial K\subset H \subset K.
$$
Then
$$
\partial K\subseteq\partial H\subseteq H\subseteq
K\subseteq\eta K=\eta H.
$$
\par If $\Omega$ is a component of $\CC\setminus H$, then $\Omega\subset K$ or $\Omega\cap K=\emptyset$.

The set $K$ can be obtained
from $H$ by filling in some holes of $H$.
\end{proposition}

\begin{remark}\label{the same connected hulls} \rm If $K\subseteq\CC$ is at most countable, then $\eta K=K$.
Therefore, for compact subsets $H,K\subseteq\CC$, if
$\eta K=\eta H$, then $H\ {\rm is\
finite \ (countable)}$ if and only if $ K\ {\rm is\ finite\ (countable)}$,
and in that case $H=K$. Particulary,  for compact subsets $H,K\subseteq\CC$, if
$\eta K=\eta H$, then $K$ is empty if and only if $H$ is empty.
\end{remark}


\begin{corollary}\label{pocetna-rub}  Let $T\in L(X)$.
 \begin{itemize}
\item[(1)] $\p\sigma_*(T)\subset \p\sigma_{TUD}(T)$, where $\sigma_*\in\{ \sigma_{B\W_+},  \sigma_{B\W},\sigma_{LD}^e, \sigma_{dsc}^e,\sigma_{B\Phi}, \sigma_{LD}, \sigma_{dsc}, \sigma_D \}$;


\item[(2)] $\p\sigma_*(T)\subset \p\sigma_{q\Phi}(T)$, where $\sigma_*\in\{\sigma_{B\W_-},\sigma_{RD}, \sigma_{RD}^e \}$.


\end{itemize}
\end{corollary}
\begin{proof}  
 Since $\sigma_{B\W_+}(T)$ is closed (Proposition \ref{closed}), it follows that $\partial\sigma_{B\W_+}(T)\subset\sigma_{B\W_+}(T)$. Hence, by using Corollary \ref{cor-W} (1), we obtain that
 $$
 \partial\sigma_{B\W_+}(T)=\partial\sigma_{B\W_+}(T)\cap  \sigma_{B\W_+}(T)= \p\sigma_{B\W_+}(T)\cap\sigma_{TUD}(T)\subset \sigma_{TUD}(T).
 $$
 Now from $\partial\sigma_{B\W_+}(T)\subset \sigma_{TUD}(T)\subset \sigma_{B\W_+}(T)$, according to Proposition \ref{pocetna}, it follows that $\partial\sigma_{B\W_+}(T)\subset \p\sigma_{TUD}(T)$.

Similarly for the rest of inclusions.
\end{proof}

It is known  that \cite[Theorem 1.65 (i)]{Ai}
$$
\partial\sigma_{\Phi}(T)\cap \acc\, \sigma_{\Phi}(T)\subset\sigma_{Kt}(T).
$$
We remark that it holds more than this: $
\partial\sigma_{\Phi}(T)\cap \acc\, \sigma_{\Phi}(T)\subset\p\sigma_{TUD}(T) $ (see Corollary \ref{cor-LDe-moved} (5) ).

Further we establish the  inclusions of the similar type for other essential  spectra.

\begin{corollary}\label{cor-W-moved}  Let $T\in L(X)$. Then

 \begin{itemize}


\item[(1)] $\partial\sigma_{\W_+}(T)\cap \acc\, \sigma_{\W_+}(T)\subset\partial\sigma_{\W_+}(T)\cap \sigma_{B\W_+}(T)\subset \partial\sigma_{TUD}(T)$;


\item[(2)] $ \partial \sigma_{\W_-}(T)\cap \acc\, \sigma_{\W_-}(T)\subset\partial\sigma_{TUD}(T)$;

\item[(3)] $ \partial \sigma_{B\W_-}(T)\cap \acc\, \sigma_{B\W_-}(T)\subset
\partial \sigma_{B\W_-}(T)\cap \acc\, \sigma_{\W_-}(T)\subset\partial\sigma_{TUD}(T)$;


\item[(4)] $\partial\sigma_{\W}(T)\cap \acc\, \sigma_{\W}(T)\subset\partial\sigma_{\W}(T)\cap \sigma_{B\W}(T)\subset \partial\sigma_{TUD}(T)$.
\end{itemize}
\end{corollary}

\begin{proof} (1) By using the first equality in Corollary \ref{cor-W} (1)\ 
 we get $$\partial \sigma_{\W_+}(T)\cap \sigma_{B\W_+}(T)=\partial \sigma_{\W_+}(T)\cap (\sigma_{TUD}(T)\cup\inter\, \sigma_{\W_+}(T))= \partial \sigma_{\W_+}(T)\cap \sigma_{TUD}(T),$$
 and therefore
\begin{equation}\label{a0}
  \partial \sigma_{\Phi_+}(T)\cap \sigma_{B\W_+}(T)\subset   \sigma_{TUD}(T).
 \end{equation}

 Let $ \lambda \in  \partial \sigma_{\W_+}(T)\cap \sigma_{B\W_+}(T)$.  Then there exists a sequence $ (\lambda_n)$ which converges to $\lambda $ and  such that $T- \lambda_n$ is upper semi-Weyl  for every $n\in\NN$. As  $T- \lambda_n$ is  upper semi-Fredholm, then it has TUD and so $\lambda_n  \notin \sigma_{TUD}(T)$, $n\in\NN$.   As $ (\lambda_n)$ converges to $\lambda $ and since $\lambda\in \sigma_{TUD}(T)$ according to \eqref{a0}, we get that $ \lambda \in \partial \sigma_{TUD}(T).$ Therefore, $\partial\sigma_{\W_+}(T)\cap \sigma_{B\W_+}(T)\subset \partial\sigma_{TUD}(T)$.

 Further, from the second equality in Corollary \ref{cor-W} (1) it follows that $\partial\sigma_{\W_+}(T)\cap \acc\, \sigma_{\W_+}(T)\subset\partial\sigma_{\W_+}(T)\cap \sigma_{B\W_+}(T)$.

(2) Let $T-\lambda I$ have TUD and let $\lambda\in \partial\sigma_{\W_-}(T)$. Since $\lambda\notin \inter\, \sigma_{\W_-}(T)$, according to Theorem \ref{F-+} we conclude that
$\lambda\notin \acc\, \sigma_{\W_-}(T)$. Therefore, $ \partial \sigma_{\W_-}(T)\cap \acc\, \sigma_{\W_-}(T)\subset\sigma_{TUD}(T)$. Now proceeding as in the proof of (1) we get $ \partial \sigma_{\W_-}(T)\cap \acc\, \sigma_{\W_-}(T)\subset\partial\sigma_{TUD}(T)$.

(3)
Suppose that $T-\lambda I$ has TUD and  $\lambda\in \partial\sigma_{B\W_-}(T)$. From  Theorem \ref{F-+} it follows that $\lambda\notin\acc\, \sigma_{\W_-}(T)$. Thus
$\partial\sigma_{B\W_-}(T) \cap acc\,   \sigma_{\W_-}(T)  \subset\sigma_{TUD}(T)$. As in  the  proof of (1), we obtain that $$\partial\sigma_{B\W_-}(T) \cap acc\,   \sigma_{\W_-}(T)  \subset\p\sigma_{TUD}(T).$$
 From $\sigma_{B\W_-}(T)\subset \sigma_{\W_-}(T)$ it follows that $\acc\, \sigma_{B\W_-}(T)\subset \acc\, \sigma_{\W_-}(T)$, which implies the first inclusion in (3).

 (4)  Similarly to the proof of (1) by using Corollary \ref{cor-W}  (3).
%
%
%
%
\end{proof}


%
%
%
%
\begin{corollary}\label{cor-LDe-moved}  Let $T\in L(X)$.
 \begin{itemize}

\item[(1)] $\partial\sigma_{\Phi_+}(T)\cap \acc\, \sigma_{\Phi_+}(T)\subset\partial\sigma_{\Phi_+}(T)\cap \sigma_{LD}^e(T)\subset \partial\sigma_{TUD}(T)$;
\item[(2)] $\partial\sigma_{\Phi_-}(T)\cap \acc\, \sigma_{\Phi_-}(T)\subset\partial\sigma_{\Phi_-}(T)\cap \sigma_{dsc}^e(T)\subset \partial\sigma_{TUD}(T)$;

\item[(3)]
$\partial \sigma_{RD}^e(T)\cap \acc\, \sigma_{RD}^e(T)\subset
\partial \sigma_{RD}^e(T)\cap \acc\, \sigma_{\Phi_-}(T)\subset \partial \sigma_{RD}^e(T)\cap \sigma_{dsc}^e(T)=\break
\partial \sigma_{RD}^e(T)\cap \sigma_{TUD}(T)\subset \partial\sigma_{TUD}(T)$;

\item[(4)] $\partial\sigma_{\Phi_-}(T)\cap  \sigma_{RD}^e(T)=\partial\sigma_{\Phi_-}(T)\cap  \sigma_{q\Phi}(T)\subset\partial\sigma_{q\Phi}(T)$;

\item[(5)] $\partial\sigma_{\Phi}(T)\cap \acc\, \sigma_{\Phi}(T)\subset\partial\sigma_{\Phi}(T)\cap \sigma_{B\Phi}(T)\subset \partial\sigma_{TUD}(T)$.

\end{itemize}
\end{corollary}
\begin{proof} It follows from Corollary \ref{cor-F}, similarly to the proof of Corollary \ref{cor-W-moved}.
\end{proof}

\begin{corollary}\label{cor-a-moved}  Let $T\in L(X)$. Then
\begin{itemize}

\item[(1)] $\partial\sigma_{ap}(T)\cap \acc\, \sigma_{ap}(T)\subset\partial\sigma_{ap}(T)\cap \sigma_{LD}(T)\subset \partial\sigma_{TUD}(T)$;

\item[(2)] $\partial\sigma_{\B_+}(T)\cap \acc\, \sigma_{\B_+}(T)\subset\partial\sigma_{\B_+}(T)\cap \sigma_{LD}(T)\subset \partial\sigma_{TUD}(T)$;

\item[(3)] $\partial\sigma_p(T)\cap \acc\, \sigma_{p}(T)\subset\partial\sigma_p(T)\cap \sigma_{LD}(T)\subset \sigma_{TUD}(T)$;
\item[(4)] $\partial\sigma_{su}(T)\cap \acc\, \sigma_{su}(T)\subset\partial\sigma_{su}(T)\cap \sigma_{dsc}(T)\subset \partial\sigma_{TUD}(T)$;

\item[(5)] $\partial \sigma_{cp}(T)\cap \acc\, \sigma_{cp}(T)\subset \partial \sigma_{cp}(T)\cap \sigma_{dsc}(T)\subset \sigma_{TUD}(T)$;

\item[(6)] $\partial \sigma_{\B_-}(T)\cap \acc\, \sigma_{\B_-}(T)\subset \partial \sigma_{\B_-}(T)\cap \sigma_{dsc}(T)\subset \p\sigma_{TUD}(T)$;

\item[(7)] $\partial \sigma_{RD}(T)\cap \acc\, \sigma_{RD}(T)\subset$ 
\ $\partial \sigma_{RD}(T)\cap \sigma_{dsc}(T)=
\partial \sigma_{RD}(T)\cap \sigma_{TUD}(T)\subset \partial\sigma_{TUD}(T)$;

\item[(8)] $\partial\sigma_{su}(T)\cap \sigma_{RD}(T)\subset \partial\sigma_{q\Phi}(T)$;

\item[(9)] $\partial \sigma_{cp}(T)\cap \sigma_{RD}(T)=\partial \sigma_{cp}(T)\cap \sigma_{q\Phi}(T)\subset \sigma_{q\Phi}(T)$;

\item[(10)] $\partial\sigma(T)\cap \acc\, \sigma(T)\subset\partial\sigma(T)\cap \sigma_D(T)\subset \partial\sigma_{TUD}(T)$;

\item[(11)] $\partial\sigma_{\B}(T)\cap \acc\, \sigma_{\B}(T)\subset\partial\sigma_{\B}(T)\cap \sigma_D(T)\subset \partial\sigma_{TUD}(T)$.

\end{itemize}
\end{corollary}
\begin{proof} It follows from  Corollary  \ref{cor-Dra}.
\end{proof}

\begin{remark} \rm  For  the operator $T$ in Example \ref{primer}, from Remark \ref{remark1},  we can conclude  that  $0\in \partial\sigma_*(T)$ where $\sigma_*\in\{\sigma_{su}, \sigma_{\Phi_-}, \sigma_{\W_-}, \sigma_{RD}, \sigma_{B\W_-}, \sigma_{RD}^e \}$. As $T$ has TUD, we have that $0\notin\sigma_{TUD}(T)$. So, in the inclusions  (2) in Corollary \ref{pocetna-rub},  as well as
in the inclusion (4)  in Corollary
\ref{cor-LDe-moved}, and
in the inclusion (8)  in Corollary \ref{cor-a-moved},   $\sigma_{q\Phi}(T)$ can not be replaced by $\sigma_{TUD}(T)$.
\end{remark}

In  the proof of the next theorem we use the following inclusions:

\bigskip

{\footnotesize
$$\begin{array}{ccccccccccc}
&&&&   \sigma_{LD}^e(T) & \subset &    \sigma_{B{\W_+}}(T) &  \subset  & \sigma_{LD}(T)& & \\
&&& \rotatebox{20}{$\subset$} & & \rotatebox{-20}{$\subset$}&
&\!\rotatebox{-20}{$\subset$}& & \rotatebox{-20}{$\subset$}&\\
 \sigma_{TUD}(T) &\subset &  \sigma_{q\Phi}(T) & \subset &
\sigma_{Kt}(T) & \subset &  \sigma_{{ B\Phi}}(T)&
  \subset  & \sigma_{{ B\W}}(T)&\subset&   \sigma_{D}(T).\\
&&& \rotatebox{-20}{$\subset$} & & \rotatebox{20}{$\subset$}&
&\rotatebox{20}{$\subset$}& & \rotatebox{20}{$\subset$}&\\
&\rotatebox{-20}{$\subset$}&&&   \sigma_{RD}^e(T)& \subset&  \sigma_{B{\W_-}}(T) & \subset & \sigma_{RD}(T) & &\\
&&& \rotatebox{20}{$\subset$} & & &
&\rotatebox{20}{$\subset$}& & &\\
&&\sigma_{dsc}^e(T)&  &\subset &  & \sigma_{dsc}(T) & & & &\\
\end{array}$$

}

\bigskip

%

\begin{theorem}\label{eta} Let  $T\in L(X)$. Then
\begin{enumerate}
\item
%
%
%

{\footnotesize
$$\begin{array}{ccccccccccc}
&&&&   \p\sigma_{LD}(T) & \subset &    \p\sigma_{B{\W_+}}(T) &  \subset  &\p \sigma_{LD}^e(T)& & \\
&&& \rotatebox{20}{$\subset$} & & \rotatebox{20}{$\subset$}& &\!\rotatebox{20}{$\subset$}& & \rotatebox{-20}{$\subset$}&\\
&& \p\sigma_{D}(T)  & \subset & \p \sigma_{{ B\W}}(T)&
  \subset  & \p\sigma_{{ B\Phi}}(T)&\subset& & &\p \sigma_{TUD}(T),\\
&&&\rotatebox{-20}{$\subset$} & & & &\rotatebox{-20}{$\subset$}& & \rotatebox{20}{$\subset$}&\\
&\ &&&   & \p\sigma_{dsc}(T)\ &   & \subset & \p\sigma_{dsc}^e(T) & &\\
\end{array}$$
}

\bigskip

{\footnotesize
$$\begin{array}{ccccccccccc}
&&&&  \p \sigma_{LD}(T) & \subset &    \p\sigma_{B{\W_+}}(T) &  \subset  & \p\sigma_{LD}^e(T)& & \\
&&& \rotatebox{20}{$\subset$} & & \rotatebox{20}{$\subset$}&
&\!\rotatebox{20}{$\subset$}& & \rotatebox{-20}{$\subset$}&\\
  & &  \p\sigma_{D}(T) & \subset &
\p\sigma_{B\W}(T) & \subset &  \p\sigma_{{ B\Phi}}(T)&
  \subset  & &&  \p \sigma_{q\Phi}(T),\\
&&& \rotatebox{-20}{$\subset$} & & \rotatebox{-20}{$\subset$}&
&\rotatebox{-20}{$\subset$}& & \rotatebox{20}{$\subset$}&\\
& &&&   \p\sigma_{RD}(T)& \subset& \p \sigma_{B{\W_-}}(T) & \subset & \p\sigma_{RD}^e(T) & &\\
\end{array}$$
}

\bigskip

{\footnotesize
$$\begin{array}{ccccccccccc}
 \p\sigma_{D}(T) &\subset &
 \p \sigma_{{ B\W}}(T)&
  \subset  & \p\sigma_{{ B\Phi}}(T)&\subset&  \p\sigma_{Kt}(T).\\
\end{array}$$
}

\item
\begin{eqnarray*}
  \eta\sigma_{TUD}(T)&=&\eta\sigma_{q\Phi}(T)=\eta\sigma_{Kt}(T)=\eta\sigma_{B\Phi}(T)=\eta\sigma_{B\W}(T)=\eta\sigma_{D}(T)\\
  &=&\eta\sigma_{LD}^e(T)=\eta\sigma_{B\W_+}(T)=\eta\sigma_{LD}(T)\\&=&\eta\sigma_{RD}^e(T)=\eta\sigma_{B\W_-}(T)=\eta\sigma_{RD}(T)\\&=&
  \eta\sigma_{dsc}^e(T)=\eta\sigma_{dsc}(T).
\end{eqnarray*}

\item
The set $\sigma_D(T)$ consists of $\sigma_{*}(T)$ and possibly some holes in $\sigma_{*}(T)$ where
 $\sigma_{*}\in\{\sigma_{q\Phi},\break\sigma_{Kt},\sigma_{B\Phi},\sigma_{B\W}, \sigma_{LD}^e, \sigma_{B{\W_+}}, \sigma_{LD},\sigma_{RD}^e, \sigma_{B\W_-}, \sigma_{RD}, \sigma_{dsc}^e, \sigma_{dsc}\}$.

\end{enumerate}
\end{theorem}
\begin{proof}

It follows from  Proposition \ref{pocetna}, the previous inclusions,
  Proposition \ref{closed} and Corollary \ref{pocetna-rub}.
   \end{proof}
From Theorem \ref{eta} and  Remark \ref{the same connected hulls} it follows that if one of $\sigma_{TUD}(T)$, $\sigma_{q\Phi}(T)$, $\sigma_{Kt}(T)$, $\sigma_{B\Phi}(T)$, $\sigma_{B\W}(T)$, $\sigma_{D}(T)$, $\sigma_{LD}^e(T)$, $\sigma_{B\W_+}(T)$, $\sigma_{LD}(T)$, $\sigma_{RD}^e(T)$, $\sigma_{B\W_-}(T)$, $\sigma_{RD}(T)$, $\sigma_{dsc}^e(T)$ and $\sigma_{dsc}(T)$ is finite (countable), then all of them are equal and therefore finite (countable). This result is already obtained in \cite[Corollary 3.4]{Q. Jiang}, but our method of proofing is different.

\begin{example}
{\em Let $Q: \ell_2(\NN) \to \ell_2(\NN)$ be the operator defined by
\[Q(\xi_1, \xi_2, \xi_3, \ldots)=(0, \xi_1, \frac{1}{2}\xi_2, \frac{1}{3}\xi_3,
\ldots), \; \; \; (\xi_1, \xi_2, \xi_3, \ldots) \in \ell_2(\NN).\]
From $\lim_{n \to \infty} ||Q^n||^{\frac{1}{n}}=\lim_{n \to \infty}
(\frac{1}{n!})^{\frac{1}{n}}=0$ we see that $Q$ is quasinilpotent.
It follows that $0$ is not an accumulation point of $\sigma_{\bf R}(Q)$ for  $\R\in\{{\bf R_i}:1\le i\le 10\}\cup\{{\bf  R^a_4, R^a_5}\}\cup\{\W_+(X),\W_-(X),\W(X),\Phi(X), \B(X), L(X)^{-1}\}$.
As $Q$ is the limit of finite rank operators, $Q$ is compact. Since $Q^n$ is compact and
$R(Q^n)$ is infinite dimensional, we conclude that $R(Q^n)$ is not
closed for every $n \in \NN$. Therefore, $Q$ is not quasi-Fredholm and so, $0\in \sigma_{qF}(Q)\subset\sigma(Q)=\{0\}$. Thus $\sigma_{qF}(Q)=\{0\}$, which implies that $\sigma_{TUD}(Q)=\{0\}$, that is $Q$ does not have TUD. It means that the condition that $T$ has TUD  can not be
omitted in Theorems \ref{F+-}, \ref{W}, \ref{SF+}, \ref{SF-}, \ref{F}, \ref{a}, \ref{su}, \ref{Dra}, \ref{final}. Also,  the condition that $T$ is quasi-Fredholm can not be
omitted in  Theorems \ref{F-+}, \ref{SF-}, \ref{su}, \ref{final}.
}
\end{example}

\section{Applications}

An operator $T\in L(X)$ is  {\it meromorphic} if its non-zero spectral points are poles of its resolvent. We say that $T$ is {\it polinomially meromorphic} if there exists non-trivial polynomial\ $p$\ \ such that $p(T)$ is meromorphic.

In \cite[Theorem 2.11]{P13} it is given a characterization of meromorphic operators in terms of B-Fredholm operators: an operator $T\in L(X)$ is meromorphic  if and only if $\sigma_{B\Phi}(T)\subset\{0\}$. This result is extended by Q. Jiang, H. Zhong and S. Zhang in \cite[Corollary 3.3]{Q. Jiang} by including   the characterization of meromorphic operators in terms of operators of topological uniform descent: $T\in L(X)$ is meromorphic  if and only if $\sigma_{TUD}(T)\subset\{0\}$. Their proof is based on the local constancy of the mappings $\lambda\mapsto K(\lambda I-T)+H_0(\lambda I-T)$ and $\lambda\mapsto \overline{K(\lambda I-T)\cap H_0(\lambda I-T)
}$ and results about SVEP established in \cite{JiangJMAA}.
  We obtain the same assertion as a corollary of Theorem \ref{eta} and our method of proofing is rather different.
\begin{theorem}
Let $T\in L(X)$. Then the following conditions are equivalent:

1. $T$ is a meromorphic operator;

2. $\sigma_{TUD}(T)\subset\{0\}$;

3. $\sigma_{B\Phi}(T)\subset\{0\}$.

%
\end{theorem}
\begin{proof} Since $T$ is a meromorphic operator if and only if $\sigma_{D}(T)\subset\{0\}$, the assertion follows from Theorem \ref{eta} (see the comment after Theorem \ref{eta}).
\end{proof}

\bigskip

For $T\in L(X)$ set $\rho_{TUD}(T)=\CC\setminus\sigma_{TUD}(T)$.

\begin{theorem}\label{omega} Let $T\in L(X)$. If $\Omega$ is a component of $\rho_{TUD}(T)$, then $\Omega\subset \sigma_D(T)$ or $\Omega\setminus E\subset\rho(T)$, where $E=\{\lambda\in\CC:\lambda\ {\rm is\ the\ pole\ of\ the \ resolvent\ od\ }T\}$.

\end{theorem}
\begin{proof} Since $\partial\sigma_D(T)\subset\sigma_{TUD}(T)$, from Proposition \ref{pocetna} it follows that
\begin{equation}\label{omega1}
 \Omega\subset \sigma_D(T) \ \ {\rm or}\ \ \Omega\cap \sigma_D(T)=\emptyset.
\end{equation}
 If the second formula in \eqref{omega1} holds, then, as $\sigma_D(T)=\sigma(T)\setminus E$, we obtain that $(\Omega\setminus E)\cap \sigma(T)=\emptyset$, which implies $\Omega\setminus E\subset\rho(T)$.
\end{proof}
In  \cite[Corollary 2.12]{Q. Jiang} Q. Jiang, H. Zhong and S. Zhang
  obtained the same result as Theorem \ref{omega} by using the  constancy of the mappings $\lambda\mapsto K(\lambda I-T)+N(\lambda I-T)$ and $\lambda\mapsto \overline{K(\lambda I-T)\cap H_0(\lambda I-T)}$ on the components of $\rho_{TUD}(T)$,
 however our proof is rather  different. They also obtained  that if $\rho_{TUD}(T)$ has only one component, then $\sigma_{TUD}(T)=\sigma_D(T)$. We get this result in a different way and also obtain that analogous assertion holds for   other spectra.
\begin{theorem}\label{only one}
Let $T\in L(X)$ and let  $\sigma_{*}\in\{\sigma_{TUD}, \sigma_{q\Phi},\sigma_{Kt},\sigma_{B\Phi},\sigma_{B\W}, \sigma_{LD}^e, \sigma_{B\W_+},\break \sigma_{LD}, \sigma_{RD}^e, \sigma_{B\W_-}, \sigma_{RD}, \sigma_{dsc}^e, \sigma_{dsc}  \}$. Then there is implication
$$
\CC\setminus\sigma_{*}(T) \ {\rm has\ only\ one\ component}\Longrightarrow \sigma_{*}(T)=\sigma_{D}(T).
$$
\end{theorem}
\begin{proof} Since
$\CC\setminus\sigma_{*}(T)$ has only one component, it follows that
$\sigma_{*}(T)$ has no holes and hence  $\sigma_{*}(T)=\eta
\sigma_{*}(T)$. According to  Theorem \ref{eta} we conclude that
$\sigma_{D}(T)\subset\eta  \sigma_{D}(T)=\eta \sigma_{*}(T)=
\sigma_{*}(T)\subset \sigma_{D}(T)$ and hence
$\sigma_{D}(T)=\sigma_{*}(T)$.
\end{proof}

\medskip

Let $F_0(X)$ denote set of finite rank operators on $X$.
Now we can prove Theorem 2.10 in \cite{P13} in a different way.
\begin{theorem} Let $T\in L(X)$ and suppose that $\sigma_{B\W}(T)$ is simply connected. Then $T+F$ satisfies the generalized version II of the Weyl's theorem for every $F\in F_0(X)$.
\end{theorem}
\begin{proof} From $F\in F_0(X)$ it follows that $\sigma_{B\W}(T)=\sigma_{B\W}(T+F)$ \cite[Theorem 4.3]{berkani-3} and  $\sigma_{B\W}(T+F)$ is simply connected. According to
Teorem \ref{eta}, since there are no holes in $\sigma_{B\W}(T+F)$, we conclude that $\sigma_D(T+F)=\sigma_{B\W}(T+F)$, and so $T+F$ satisfies the generalized version II of the Weyl's theorem.
\end{proof}

\begin{corollary}\label{prazan} 
Let $T\in L(X)$.
\begin{itemize}

\item[(1)] If $\sigma_p(T)\subset \partial\sigma_p(T)$,  then $\sigma_{LD}(T)=\sigma_{TUD}(T)$.

 \item[(2)]  If $\sigma_{cp}(T)\subset\partial \sigma_{cp}(T)$, then $\sigma_{dsc}(T)=\sigma_{TUD}(T)$ and $\sigma_{RD}(T)=\sigma_{q\Phi}(T)$.

\item[(3)] If $\sigma_*(T)=\partial \sigma_*(T)$, where $\sigma_*\in\{ \sigma_{D}, \sigma_{LD}, \sigma_{LD}^e,\sigma_{B\Phi}, \sigma_{B\W_+},\sigma_{B\W} \}$, then
$\sigma_{*}(T)=\sigma_{TUD}(T)$.

\item[(4)] If $\sigma_{dsc}(T)=\partial \sigma_{dsc}(T)$,  then $\sigma_{dsc}(T)=\sigma_{TUD}(T)$ and $\sigma_{RD}(T)=\sigma_{q\Phi}(T)$.


\item[(5)]  If $\sigma_{dsc}^e(T)=\partial \sigma_{dsc}^e(T)$, then $\sigma_{dsc}^e(T)=\sigma_{TUD}(T)$ and $\sigma_{RD}^e(T)=\sigma_{q\Phi}(T)$.
%
%

\item[(6)] If  $\sigma_{B\W_-}(T)=\partial \sigma_{B\W_-}(T)$,  then $\sigma_{B\W_-}(T)=\sigma_{q\Phi}(T)$.


\end{itemize}

\end{corollary}
\begin{proof} (1) From $\sigma_p(T)\subset \partial\sigma_p(T)$ it follows that  $\inter\, \sigma_p(T)=\emptyset$. Using Corollary \ref{cor-Dra} (1) we get $\sigma_{LD}(T)=\sigma_{TUD}(T)$.

The rest of  assertions  can be proved similarly by using Corollaries  \ref{cor-Dra},  \ref{cor-F}  and \ref{cor-W}.
\end{proof}

\begin{remark}\label{poslednja} \rm   Q. Jiang, H. Zhong and S. Zhang concluded in \cite[p. 1156]{Q. Jiang} that if $\sigma(T)$ is  contained in a line segment, then $\sigma_{D}(T)=\sigma_{TUD}(T)$. From Corollary \ref{prazan} (3) we get  that if $\sigma(T)$ is  contained in a line, then
$\sigma_{D}(T)=\sigma_{TUD}(T)$. If $T$ is
unitary operator on Hilbert space, then its spectrum is contained in a
line and so, $\sigma_{D}(T)=\sigma_{TUD}(T)$.

From  Corollary \ref{prazan} it follows also that  if $\sigma_{*}(T)$ is  contained in a line for  $\sigma_{*}\in\{\sigma_{LD},  \sigma_{D},\break \sigma_{LD}^e, \sigma_{B\W_+}, \sigma_{B\W}, \sigma_{B\Phi},     \sigma_{dsc}^e, \sigma_{dsc}  \}$,
    then $\sigma_{*}(T)=\sigma_{TUD}(T)$. Also, if $\sigma_{*}(T)$ is  contained in a line for
     $\sigma_{*}\in\{  \sigma_{RD}^e, \sigma_{B\W_-}, \sigma_{RD}\}$, then $\sigma_{*}(T)=\sigma_{qF}(T)$.

  Therefore,  if $\sigma_{\bf R}(T)$ is  contained in a line for
     $\R\in\{{\bf R_6, R_7,R_8,R_9,R_{10}, R^a_4, R^a_5},\break \W_+(X),\W(X),\Phi(X), \B(X), L(X)^{-1}\}$, then
$\sigma_{\bf{BR}}(T)=\sigma_{TUD}(T)$. If $\R\in\{{\bf R_1,R_2,R_3,R_4},\break {\bf R_5},\W_-(X)\}$ and  $\sigma_{\bf R}(T)$ is  contained in a line, then $\sigma_{\bf{BR}}(T)=\sigma_{qF}(T)$ (see also Theorem \ref{final}).

Furthermore, if $\sigma_{dsc}^e(T)$ ($\sigma_{dsc}(T)$)  is  contained in a line, then  $\sigma_{RD}^e(T)=\sigma_{qF}(T)$ ($\sigma_{RD}(T)=\sigma_{qF}(T)$).
If   $\sigma_p(T)$ ($\sigma_{cp}(T)$) is countable or contained in a line, then  $\sigma_{LD}(T)=\sigma_{TUD}(T)$ ($\sigma_{RD}(T)=\sigma_{q\Phi}(T)$ and $\sigma_{dsc}(T)=\sigma_{TUD}(T)$).
\end{remark}
 \begin{example} \rm
 If $X$ is one of $c_{0}(\ZZ)$ and $\ell_{p}(\ZZ)$, $p\ge 1$,
then for the forward and backward bilateral shifts $W_1,\ W_2\in
L(X)$ there are equalities
$
  \sigma (W_1)=\sigma (W_2)=\partial\DD$, where $\DD=\{\lambda\in\CC:|\lambda|\le 1\}$.
 From $\acc\, \sigma (W_1)=\acc\, \sigma (W_2)=\partial\DD$ we conclude that $\sigma_D (W_1)=\sigma_D (W_2)=\partial\DD$ and since $\sigma (W_1)$ and $\sigma (W_2)$ are contained in a line, we obtain $\sigma_{TUD} (W_1)=\sigma_{TUD} (W_2)=\partial\DD$.
\end{example}




\begin{corollary}\label{skoro}
Let $T\in L(X)$.

 If
 $\sigma_{*}\in\{\sigma_{\W_+}, \sigma_{\W_-},\sigma_\W, \sigma_{B\W_-}, \sigma_{{\Phi_+}},  \sigma_{{\Phi_-}},  \sigma_{\Phi},\sigma_{RD}^e, \sigma_{ap},\sigma_{su},\sigma_{\B_+}, \sigma_{\B_-}, \sigma_{\B},\sigma_{RD}, \sigma\}$
   and
\begin{equation}\label{bac}
\sigma_{*}(T)=\partial\sigma_{*}(T)=\acc\,
\sigma_{*}(T),
\end{equation}
 then
 \begin{equation}\label{solo}
\sigma_{TUD}(T)=
\sigma_{*}(T).
\end{equation}
\end{corollary}
\begin{proof} From Corollaries \ref{cor-W-moved},  \ref{cor-LDe-moved} and  \ref{cor-a-moved} we have that $\partial \sigma_{*}(T)\cap\acc\, \sigma_{*}(T)\subset \sigma_{TUD}(T)$, which together with the equalities \eqref{bac} gives the inclusion $\sigma_{*}(T)\subset\sigma_{TUD}(T)$. Since
 $\sigma_{TUD}(T)\subset
\sigma_{*}(T)$, we get \eqref{solo}.
\end{proof}

We recall that if $K\subset\CC$ is compact, then
for $\lambda\in\partial K$  the following equivalence holds:
\begin{equation}\label{zz}
 \lambda\in \acc\, K\Longleftrightarrow \lambda\in \acc\, \partial K.
\end{equation}
The following corollary is an improvement of Theorem 2.10 and Corollary 2.11 in \cite{aienarosas}.

\begin{corollary}\label{isom} Let $T\in L(X)$.
\begin{enumerate}

\item
 Let $T$ be an operator for which $\sigma_{ap}(T)=\partial\sigma (T)$ and every $\lambda\in \partial\sigma (T)$ is not isolated in $\sigma (T)$. Then
$\sigma_{ap}(T)=\sigma_{TUD}(T)
$.

\item
 Let $T$ be an operator for which $\sigma_{su}(T)=\partial\sigma (T)$ and every $\lambda\in \partial\sigma (T)$ is not isolated in $\sigma (T)$. Then
$\sigma_{su}(T)=\sigma_{TUD}(T)
$.

\end{enumerate}
\end{corollary}
\begin{proof}
From $\sigma_{ap}(T)=\partial\sigma (T)$ and $\partial\sigma(T)\subset\partial\sigma_{ap}(T)\subset\sigma_{ap}(T)$   it follows
that $\sigma_{ap}(T)=\partial\sigma_{ap}(T)$, while from (\ref{zz})
it follows that every $\lambda\in \partial\sigma (T)$ is not
isolated in $\partial\sigma (T)$. Therefore, every $\lambda\in \partial\sigma (T)$ is not
isolated  in $\sigma_{ap}(T)$ and hence, $\sigma_{ap}(T)\subset\acc\, \sigma_{ap}(T)$. Thus $\sigma_{ap}(T)=\partial\sigma_{ap}(T)=\acc\,
\sigma_{ap}(T)$ and from Corollary \ref{skoro} it follows that $\sigma_{ap}(T)=\sigma_{TUD}(T)
$.

The assertion (2) can be proved similarly.
\end{proof}
\begin{example} \rm For  each  $X\in \{  c_0(\NN), c(\NN), \ell_{\infty}(\NN),
\ell_p(\NN)\}$, $p\ge 1$, and the forward and backward unilateral shifts
$U$, $V\in L(X)$ there are equalities $\sigma(U)=\sigma(V)=\DD$, $\sigma_D(U)=\sigma_D(V)=\DD$ and  $
\sigma_{ap}(U)=\sigma_{su}(V)=\partial\DD$. By using Corollary \ref{skoro} (or Corollary \ref{isom}) we obtain that $\sigma_{TUD}(U)=\sigma_{ap}(U)=\partial\DD$ and $\sigma_{TUD}(V)=\sigma_{su}(V)=\partial\DD$. It implies that
$$\sigma_{TUD}(U)=\sigma_{q\Phi}(U)=\sigma_{LD}^e(U)=\sigma_{B\W_+}(U)=\sigma_{LD}(U)=\sigma_{ap}(U)=\partial\DD$$
and
$$\sigma_{TUD}(V)\! =\! \sigma_{q\Phi}(V)\! =\! \sigma_{RD}^e(V)\! =\! \sigma_{B\W_-}(V)\! =\! \sigma_{RD}(V)\! =\! \sigma_{dsc}(V)\! =\! \sigma_{dsc}^e(V)\! =\! \sigma_{su}(V)\! =\! \partial\DD.$$

 As $\sigma_{su}(U)=\DD$, from Corollary \ref{cor-Dra} ((2), (3)) it follows that $\sigma_{dsc}(U)=\sigma_{RD}(U)=\DD$.
 Since $\sigma_{\Phi}(U)=\sigma_{\Phi}(V)=\p\DD$ \cite[Theorem 4.2]{ZDH}, from Remark  \ref{poslednja}  we get that $\sigma_{B\Phi}(U)=\sigma_{TUD}(U)=\partial\DD$ and similarly,
 $\sigma_{B\Phi}(V)=\p\DD$. From the inclusions $\sigma_{TUD}(U)\subset\sigma_{Kt}(U)\subset \sigma_{B\Phi}(U)$ we have that $\sigma_{Kt}(U)=\p\DD$ and similarly, $\sigma_{Kt}(V)=\p\DD$.
%
From $\p\sigma_{\Phi}(U)\subset\sigma_{\Phi_-}(U)\subset\sigma_{\Phi}(U)$ it follows  that   $\sigma_{\Phi_-}(U)=\p\DD$, that is $\sigma_{\Phi_-}(U)$ is contained in the line and hence, by Remark  \ref{poslednja} 
we get   that $\sigma_{dsc}^e(U)=\sigma_{TUD}(U)=\p\DD$ and $\sigma_{RD}^e(U)=\sigma_{q\Phi}(U)=\p\DD$.

 From $\p\sigma_{\Phi}(V)\subset\sigma_{\Phi_+}(V)\subset\sigma_{\Phi}(V)$ we conclude that   $\sigma_{\Phi_+}(V)=\p\DD$ which according to Remark  \ref{poslednja} implies that $\sigma_{LD}^e(V)=\sigma_{TUD}(V)=\p\DD$. As $\sigma_{RD}(V)=\p\DD$ and $\sigma_D(V)=\DD$, we get $\sigma_{LD}(V)=\DD$.
 From $\sigma_{\Phi}(V)=\p\DD$, $\sigma_{ap}(V)=\DD$ and $\sigma_{su}(V)=\p\DD$, we conclude that for   $|\lambda|<1$ it holds that $V-\lambda I$ is Fredholm with positive  index and so, $\{\lambda\in\CC:|\lambda|<1\}\subset\sigma_{\W_+}(V)\subset\sigma_{\W}(V)\subset\DD$, which implies that $\sigma_{\W_+}(V)=\sigma_{\W}(V)=\DD$. From Corollary \ref{cor-W} (1), (3) it follows that $\sigma_{B\W_+}(V)=\DD$ and  $\sigma_{B\W}(V)=\DD$. Similarly, from $\sigma_{\Phi}(U)=\p\DD$, $\sigma_{ap}(U)=\p\DD$ and $\sigma_{su}(U)=\DD$ it follows that $\sigma_{\W_-}(U)=\sigma_{\W}(U)=\DD$, which by Corollary \ref{cor-W} (2), (3) implies that $\sigma_{B\W_-}(U)=\sigma_{B\W}(U)=\DD$.
\end{example}
For $\lambda_0\in \CC$, the open disc, centered  at $\lambda_0$  with radius $\epsilon$ in $\CC$, is denoted  by $D(\lambda_0,\epsilon)$.
\begin{example} \rm  Every  non-invertible isometry $T$ has the property that $\sigma(T)=\overline{D(0,r(T))}$ and $\sigma_{ap}(T)=\partial D(0,r(T))$, where $r(T)$ is  the spectral radius of $T$ \cite[p. 187]{aienarosas}. Hence   $\sigma_{ap}(T)=\partial\sigma (T)$ and every $\lambda\in \partial\sigma (T)$ is not isolated in $\sigma (T)$. Therefore, according to Corollary  \ref{isom}, for arbitrary non-invertible isometry $T$ we get that
$  \sigma_{TUD}(T)=\sigma_{q\Phi}(T)=\sigma_{LD}^e(T)=\sigma_{B\W_+}(T)=\sigma_{LD}(T)=\sigma_{ap}(T)= \partial D(0,r(T))$.
\end{example}

\begin{example}{\em For the $Ces\acute{a}ro\ operator$ $C_p$ defined on the classical Hardy space $H_p(\DDD)$, $\DDD$ the open unit disc and $1<p<\infty$, by
$$
 (C_pf)(\lambda)=\ds\frac 1{\lambda}\int_0^{\lambda}\ds\frac{f(\mu)}{1-\mu}\, d\mu,\ \, {\rm for\ all\ }f\in H_p(\DDD)\ {\rm and\ }\lambda\in\DDD,
 $$
it is known that its spectrum is  the closed disc $\Gamma_p$ centered at $p/2$ with radius $p/2$, $\sigma_{Kt}(C_p)=\sigma_{ap}(C_p)=\partial \Gamma_p$ and also $\sigma_{\Phi}(C_p)=\partial \Gamma_p$ \cite{Mill}, \cite{aienarosas}.
According to Corollary \ref{skoro} or Corollary \ref{isom} we get that $\sigma_{TUD}(C_p)=\sigma_{q\Phi}(C_p)=\sigma_{LD}^e(C_p)=\sigma_{B\W_+}(C_p)=\sigma_{LD}(C_p)=\sigma_{ap}(C_p)=\partial \Gamma_p$.
Since $\sigma_{\Phi}(C_p)$ is contained in the line  and hence also $\sigma_{\Phi_+}(C_p)$ and $\sigma_{\Phi_-}(C_p)$ are  contained in the line, according to Remark \ref{poslednja} we conclude that $\sigma_{B\Phi} (C_p)=\sigma_{dsc} ^e(C_p)=\sigma_{TUD}(C_p)=\partial \Gamma_p$ and $\sigma_{RD} ^e(C_p)=\sigma_{qF}(C_p)=\partial \Gamma_p$. From
 $\sigma(C_p)= \Gamma_p$ and  $\sigma_{ap}(C_p)=\partial \Gamma_p$ it follows that  and $\sigma_{su}(C_p)= \Gamma_p$ which together with $\sigma_{\Phi}(C_p)=\partial \Gamma_p$ implies that $\sigma_{\W_-}(C_p)=\sigma_{\W}(C_p)=\Gamma_p$. Now from Corollary \ref{cor-W} (2), (3) we obtain that $\sigma_{B\W_-}(C_p)=\sigma_{B\W}(C_p)=\Gamma_p$. As $\sigma_{D}(C_p)= \Gamma_p$ and  $\sigma_{LD}(C_p)=\partial \Gamma_p$, it follows that $\sigma_{RD}(C_p)= \Gamma_p$, which by  Corollary \ref{cor-Dra} (2) implies that $\sigma_{dsc}(C_p)= \Gamma_p$.
}
\end{example}

%
%
%
%
%
%
%
%

\noindent
\author{Sne\v zana
\v C. \v Zivkovi\'c-Zlatanovi\'c}

\noindent{University of Ni\v s\\
Faculty of Sciences and Mathematics\\
P.O. Box 224, 18000 Ni\v s, Serbia}

\noindent {\it E-mail}: {\tt mladvlad@mts.rs}

\bigskip
\noindent
\author{Mohammed Berkani}

\noindent Department of Mathematics,\\
 \noindent Science faculty of Oujda,\\
\noindent University Mohammed I,\\
\noindent Laboratory LAGA, \\
\noindent Morocco\\
\noindent berkanimo@aim.com,\\


\begin{thebibliography}{99}

\bibitem{Ai} P. Aiena, {\em Fredholm and local spectral theory,
with applications to multipliers}, Kluwer Academic Publishers
(2004).

\bibitem{aienarosas} P. Aiena, E. Rosas, {\em Single-valued
extension property at the points of the approximate point spectrum},
J. Math. Anal. Appl. 279 (2003), 180-188.

\bibitem{aienatriolo}  P. Aiena, S. Triolo, {\em
Local spectral theory for Drazin invertible operators}, J. Math. Anal. Appl.
435(1) (2016),  414-424.

\bibitem{aienatriolo2}  P. Aiena, S. Triolo, {\em Fredholm Spectra and Weyl Type Theorems for Drazin Invertible Operators},
Mediterranean Journal of Mathematics
13(6) (2016),  4385-4400.

\bibitem{P7}{M. Berkani},
{\em On a class of quasi-Fredholm operators}, Integr. Equ. Oper.
Theory, 34 (1999), 244-249.




\bibitem{P8}{M. Berkani},
{\em Restriction of an operator to the range of its powers},
Studia Mathematica, 140 (2)\, (2000), 163-174.

\bibitem{BerkaniSarih} {M.Berkani, M. Sarih}, {\em On semi B-Frdholm operators}, Glasgow Math. J. 43 (2001) 457-465.

\bibitem{berkani-3}{M. Berkani},
{\em Index of B-Fredholm operators and generalization of a Weyl
Theorem\/}, Proc. Amer. Math. Soc., 130 (6)\, (2002), 1717-1723.

\bibitem{P13} {M. Berkani}, {\em B-Weyl spectrum and poles of the
resolvent}, J. Math. Anal. App. 272 (2002), 596-603.

\bibitem{BCD} M. Berkani, N. Castro, S.V. Djordjevi\'c, {\em Single valued extension property and generalized Weyl’s theorem}, Math. Bohem. 131 (1) (2006) 29–38.

\bibitem{BKMO} M. Burgos, A. Kaidi, M. Mbekhta and M. Oudghiri, {\em The descent spectrum and perturbations}, J. Operator Theory, 56:2 (2006), 259-271.


\bibitem{Fredj} O. Bel Hadj Fredj, {\em Essential  descent spectrum and commuting compact perturbations}, Extracta Math. 21(3),  (2006), 261-271.


\bibitem{Grabiner}  S. Grabiner, {\it Uniform ascent and descent of bounded operators}; J. Math. Soc.
Japan 34, no 2, (1982), 317-337.


\bibitem{GZ} { S. Grabiner and J. Zem\' anek} {\em Ascent, descent, and ergodic properties of linear operators}, J. Operator Theory, 48 (2002), 69-81.





\bibitem{H}R.E. Harte, {\it Invertibility and Singularity for Bounded Linear Operators}, Marcel
Dekker, Inc., New York/Basel, 1988.

\bibitem{HW} R.E. Harte and A.W. Wickstead, {\it Boundaries, hulls
and spectral mapping theorems}, Proc. Royal Irish Acad. 81A (1981)
201-208.



\bibitem{JiangJMAA} Q. Jiang, H. Zhong, Q. Zeng, {\em Topological uniform descent and Localized SVEP}, J. Math. Anal. Appl. 390 (2012) 355–361.


\bibitem{Q. Jiang} Q. Jiang, H. Zhong and S. Zhang, {\em Components of topological uniform descent resolvent set
and local spectral theory}, Linear Algebra Appl., 438 (2013), 1149-1158.

\bibitem{Kaashoek} M. A. Kaashoek, \textit{Ascent, descent, nullity and defect, a note on a
paper by A. E. Taylor}, Math. Annalen, {\bf 172} (1967), 105--115.


\bibitem{KM} V. Kordula and V.  M\"uller, {\em On the axiomatic theory of spectrum}, Studia Math. ibid. 119 (1996), 109-128.


\bibitem{MbekhtaMuller} M. Mbekhta and V.  M\"uller, {\em On the axiomatic theory of spectrum II}, Studia Math. 119(2) (1996),  129-147.


\bibitem{Mill} T. L. Miller, V. G. Miller, R. C. Smith, {\it Bishop's property ($\beta$) and $Ces\acute{a}ro\ operator$}, J. London Math. Soc. (2)58 (1998) 197-207.



\bibitem {MV} D. Mili\v ci\'c and K. Veseli\'c, {\it On the boundary
of essential spectra}, Glasnik Mat. tom 6 (26) No 1 (1971), 73-78.


\bibitem{Mu} V. M\"uller, {\em Spectral theory of linear operators and spectral
systems in Banach algebras}, Birkh\"auser (2007).

\bibitem{V1}  V. Rako\v cevi\'c, {\it On one subset of
M. Schechter's essential spectrum}, Mat.Vesnik, {\bf 33}(1981),
389-391.

\bibitem {V2} V. Rako\v cevi\' c, {\it
Approximate point spectrum and commuting comact perturbations},
Glasgow Math. J., {\bf 28}(1986), 193-198.

 \bibitem {ZDH} S. \v C. \v Zivkovi\'c-Zlatanovi\'c, D. S.
Djordjevi\'c and R.E. Harte, {\it Polynomially Riesz
perturbations}, J. Math. Anal. Appl. 408 (2013) 442–451.

\end{thebibliography}
\end{document}